\documentclass{amsart}
\usepackage{eurosym}
\usepackage{amssymb}
\usepackage{amsmath}
\usepackage{amsfonts}
\usepackage{graphicx,color}

\newtheorem{theorem}{Theorem}
\theoremstyle{plain}

\newtheorem{lemma}{Lemma}

\numberwithin{equation}{section}

\begin{document}
\title[An elliptic system with
logarithmic nonlinearity]{An elliptic system with
logarithmic nonlinearity}

\begin{abstract}
In the present paper we study the existence of solutions for some classes of
singular system involving the $\Delta _{p(x)}$ and $\Delta _{q(x)}$
Laplacian operators. The approach is based on bifurcation theory and
sub-supersolution method for systems of quasilinear equations involving
singular terms.
\end{abstract}

\author{Claudianor O. Alves}
\address{Claudianor O. Alves -
Unidade Acad\^{e}mica de Matem\'{a}tica \\
Universidade Federal de Campina Grande\\
Av. Aprigio Veloso, 882 \\
CEP:58429-900, Campina Grande - PB \\
Brazil}
\email{coalves@mat.ufcg.edu.br}
\thanks{C.O. Alves was partially supported by CNPq/Brazil 304036/2013-7 and
INCT-MAT}
\author{Abdelkrim Moussaoui}
\address{Abdelkrim Moussaoui -
Biology Department \\
A. Mira Bejaia University \\
Targa Ouzemour, 06000 Bejaia \\
Algeria}
\email{abdelkrim.moussaoui@univ-bejaia.dz}
\thanks{A. Moussaoui was supported by CNPq/Brazil 402792/2015-7.}
\author{Leandro da S. Tavares}
\address{Leandro da S. Tavares -
Universidade Federal do Cariri \\
Av. Ten. Raimundo Rocha s/n \\
CEP:63048-080, Juazeiro do Norte - CE
\\
Brazil}
\email{lean.mat.ufca@gmail.com}
\maketitle

\section{Introduction and statement of the main results}

\label{Section-2}

Let $\Omega \subset {\mathbb{R}}^{N}$ $(N\geq 2)$ be a bounded domain with
smooth boundary $\partial \Omega $. We are interested in the following
quasilinear system%
\begin{equation}
\left\{ 
\begin{array}{ll}
-\Delta _{p(x)}u=-\gamma \log v+\theta v^{\alpha (x)} & \ \mbox{in}\ \Omega ,
\\ 
-\Delta _{q(x)}v=-\gamma \log u+\theta u^{\beta (x)} & \ \mbox{in}\ \Omega ,
\\ 
u,v>0 & \ \mbox{in}\ \Omega , \\ 
u=v=0 & \ \mbox{on}\ \partial \Omega ,%
\end{array}%
\right.  \label{p}
\end{equation}%
which exhibits a singularity at zero through logarithm function. The
variable exponents $\alpha (.),\beta (.)$ are positive, the constants $%
\gamma ,\theta >0$ and $\Delta _{p(x)}$ (resp. $\Delta _{q(x)}$) stands for
the $p(x)$-Laplacian (resp. $q(x)$-Laplacian) differential operator on $%
W_{0}^{1,p(x)}(\Omega )$ (resp. $W_{0}^{1,q(x)}(\Omega )$) with $p,q\in
C^{1}(\overline{\Omega }),$%
\begin{equation}
p^{\prime }(x)\leq p(x)^{\ast }\text{, }q^{\prime }(x)\leq q(x)^{\ast }\text{
\ and }\left\{ 
\begin{array}{l}
1<p^{-}\leq p^{+}<N \\ 
1<q^{-}\leq q^{+}<N,%
\end{array}%
\right.  \label{H1}
\end{equation}%
where $p(x)^{\ast }=\frac{Np(x)}{N-p(x)}$ and $q(x)^{\ast }=\frac{Nq(x)}{%
N-q(x)}$. In the sequel we denote by%
\begin{equation*}
\begin{array}{l}
s^{-}=\inf_{x\in \Omega }s(x)\text{ \ and \ }s^{+}=\sup_{x\in \Omega }s(x)%
\text{ for }s\in C(\overline{\Omega }).%
\end{array}%
\end{equation*}

Throughout this paper, we denote by $\mathcal{M}\subset C^{1}(\overline{%
\Omega })\times C^{1}(\overline{\Omega })$ the pair of functions $(u,v)\in
C^{1}(\overline{\Omega })\times C^{1}(\overline{\Omega })$ such that there
is a constant $c>0$, which depends on $u$ and $v$, verifying 
\begin{equation}\label{dist-prop}
u(x),v(x)\geq cd(x)\text{ \ in }\Omega ,
\end{equation}%
where $d(x):=dist(x,\partial \Omega ).$

A weak solution of (\ref{p}) is a pair $(u,v)\in W_{0}^{1,p(x)}(\Omega
)\times W_{0}^{1,q(x)}(\Omega ),$ with $u,v$ being positive a.e. in $\Omega $
and satisfying%
\begin{equation*}
\left\{ 
\begin{array}{c}
\int_{\Omega }|\nabla u|^{p(x)-2}\nabla u\nabla \phi \ dx=\int_{\Omega
}(-\gamma \log v+\theta v^{\alpha (x)})\phi \ dx \\ 
\int_{\Omega }|\nabla v|^{q(x)-2}\nabla v\nabla \psi \ dx=\int_{\Omega
}(-\gamma \log u+\theta u^{\beta (x)})\psi \ dx,%
\end{array}%
\right.
\end{equation*}%
for all $(\phi ,\psi )\in W_{0}^{1,p(x)}(\Omega )\times
W_{0}^{1,q(x)}(\Omega ).$\bigskip

The study of problems involving variable exponents growth conditions is
widely justified with many physical examples and arise from a variety of
nonlinear phenomena. They are used in electrorheological fluids as well as
in image restorations. For more inquiries on modeling physical phenomena
involving $p(x)$-growth condition we refer to \cite%
{Acerbi1,Acerbi2,Antontsev,CLions,Chen,MR, R*, RR*,Ru,Re*}.

Elliptic problems involving the logarithmic nonlinearity appear in some
physical models like in dynamic of thin films of viscous fluids, see for
instance \cite{Ferreira-Queiroz}. An interesting point regarding these
problems comes out from the fact that $-\log x$ is of sign changing and
behaving at the origin like the power function $t^{\alpha }$ for $\alpha <0$
with a slow growth. In addition, the logarithmic function is not invariant
by scaling which does not occur with the power function. These facts
motivated the recent studies in \cite{MQ}, \cite{Queiroz} and \cite%
{Ferreira-Queiroz}, where the authors considered the scalar semilinear case
of (\ref{p}) (that is, $p(x)=q(x)=2$) with constant exponents and by
essentially using the linearity of the principal part. We also mention \cite%
{MY}, focusing on problem with constant exponents involving nonlinear
operator.

The essential point in this work is that the singularity in system (\ref{p})
comes out through logarithmic nonlinearities involving variable exponents
growth conditions. According to our knowledge, it is for the first time when
such problems are studied. Our main results provide the existence and
regularity of (positive) solutions for problem (\ref{p}). They are stated as
follows.

\begin{theorem}
\label{T1}Assume (\ref{H1}) holds. Then

\begin{description}
\item[\textrm{(i)}] If 
\begin{equation}
0<\alpha ^{-}\leq \alpha ^{+}<q^{-}-1,\text{ \ }0<\beta ^{-}\leq \beta
^{+}<p^{-}-1  \label{H2}
\end{equation}%
problem (\ref{p}) has a solution $(u,v)\in \mathcal{M}$ for all $\theta
,\gamma >0$.

\item[\textrm{(ii)}] If 
\begin{equation}
\alpha ^{-}>q^{+}-1,\text{ \ }\beta ^{-}>p^{+}-1,  \label{hip-superlinear}
\end{equation}%
problem (\ref{p}) has a solution $(u,v)\in \mathcal{M}$ for $\gamma $ small
enough and for all $\theta >0$.

\item[\textrm{(iii)}] If 
\begin{equation}
\alpha ^{+}>q^{-}-1\ \ \text{and \ }\beta ^{+}>p^{-}-1,  \label{c}
\end{equation}%
problem (\ref{p}) admits a solution $(u,v)\in \mathcal{M}$ for $\gamma $ and 
$\theta $ small enough.
\end{description}
\end{theorem}

\bigskip

\begin{theorem}\label{T2}Assume \eqref{H1} and that \begin{equation}
\left\{ 
\begin{array}{l}
\frac{\gamma }{\theta e}<\alpha ^{-}\leq \alpha ^{+}<\min \{p^{-}-1,\frac{%
q^{\prime }(x)}{p^{\prime }(x)}\} \\ 
\frac{\gamma }{\theta e}<\beta ^{-}\leq \beta ^{+}<\min \{q^{-}-1,\frac{%
p^{\prime }(x)}{q^{\prime }(x)}\}%
\end{array}%
\right.  \label{H2'}
\end{equation}%
holds for all $x\in \Omega $. Then problem (\ref{p}) has a positive solution 
$(u,v)\in W_{0}^{1,p(x)}(\Omega )\times W_{0}^{1,q(x)}(\Omega )$ satisfying \eqref{dist-prop}.
\end{theorem}

The proof of Theorem \ref{T1} is done in section \ref{S4}. Our approach
relies on the sub-supersolutions techniques. However, this method in its
system version (see \cite[p. 269]{carl}) does not work for problem (\ref{p})
due to its noncooperative character, which means that the right hand sides
of the equations in (\ref{p}) are not necessarily increasing whenever $u$
(resp. $v$) is fixed in the first (resp. second) equation in (\ref{p}).
Another reason this approach cannot be directly implemented is the presence
of singularities in (\ref{p}). To overcome this difficulties, we disturb
problem (\ref{p}) by introducing a parameter $\varepsilon >0$. This gives
rise to a regularized system for (\ref{p}), depending on $\varepsilon >0,$
whose study is relevant for our initial problem. We construct a
sub-supersolution pair for the regularized system, independent on $%
\varepsilon ,$ and we show the existence of positive family of solutions $%
(u_{\varepsilon },v_{\varepsilon })\in C^{1,\gamma }(\overline{\Omega }%
)\times C^{1,\gamma }(\overline{\Omega })$, for certain $\gamma \in (0,1)$,
through a new result regarding sub-supersolutions for quasilinear
competitive (noncooperative) systems involving variable exponents growth
conditions (see section \ref{S3}). Then, a (positive) solution of (\ref{p})
is obtained by passing to the limit as $\varepsilon \rightarrow 0$
essentially relying on the independence on $\varepsilon $ of the upper and
lower bounds of the approximate solutions $(u_{\varepsilon },v_{\varepsilon
})$ and on Arzel\`{a}-Ascoli's Theorem. An important part of our result lies
in the obtaining of the sub and supersolution which cannot be constructed
easily. Precisely, this is due to the fact that $p(x)$-Laplacian opeartor is
inhomogeneous and in general, it has no first eigenvalue, that is, the
infimum of the eigenvalues of $p(x)$-Laplacian equals $0$ (see \cite{FZZ2}).
At this point, the choice of suitable functions with an adjustment of
adequate constants is crucial.

The proof of Theorem \ref{T2} is done in section \ref{S5}. It is chiefly
based on \textcolor{red}{a}  Theorem by Rabinowitz (see \cite{Rabinowitz}) which establishes,
for each $\varepsilon >0$, the existence of positive solutions $%
(u_{\varepsilon },v_{\varepsilon })$ for the regularized problem of (\ref{p}%
) in $W_{0}^{1,p(x)}(\Omega )\times W_{0}^{1,q(x)}(\Omega )$. The solution
of (\ref{p}) under assumption (\ref{H2'}) is obtained by passing to the
limit as $\varepsilon \rightarrow 0$. This is based on a priori estimates,
Hardy-Sobolev Inequality, and Lebesgue's dominated convergence Theorem.

A significant feature of our existence results concerns the regularity part.
In Theorem \ref{T1}  the regularity of the obtained solution
for problem (\ref{p}) is derived through the weak comparison principle and
the regularity result in \cite{Alves-Moussaoui}.

\section{Preliminaries}

\label{S2}

Let $p\in C(\overline{\Omega })$ with $p(x)>1$ in $\Omega $. Consider the
Lebesgue's space 
\begin{equation*}
L^{p(x)}(\Omega ):=\left\{ u:\Omega \rightarrow \mathbb{R}\,:\,u\ \text{is a
measurable and}\,\int_{\Omega }|u(x)|^{p(x)}\ dx<+\infty \right\} 
\end{equation*}%
which is a Banach space with the Luxemburg norm 
\begin{equation*}
\Vert u\Vert _{L^{p(x)}(\Omega )}:=\left\{ \lambda >0;\int_{\Omega
}\left\vert \frac{u(x)}{\lambda }\right\vert ^{p(x)}\ dx\leq 1\right\} .
\end{equation*}%
The Banach space $W^{1,p(x)}(\Omega )$ is defined as 
\begin{equation*}
W^{1,p(x)}(\Omega ):=\{u\in L^{p(x)}(\Omega );|\nabla u|\in L^{p(x)}(\Omega
)\},
\end{equation*}%
equipped with the norm 
\begin{equation*}
\Vert u\Vert _{W^{1,p(x)}(\Omega )}:=\Vert u\Vert _{L^{p(x)}(\Omega )}+\Vert
\nabla u\Vert _{L^{p(x)}(\Omega )}.
\end{equation*}%
The space $W_{0}^{1,p(x)}(\Omega )$ is defined as closure of $C_{0}^{\infty
}(\Omega )$ in $W^{1,p(x)}(\Omega )$ with respect to the norm. The space $%
W_{0}^{1,p(x)}(\Omega )$ is separable and reflexive Banach spaces when $%
p^{-}>1$. For a later use, we recall that the embedding 
\begin{equation}
\begin{array}{l}
W_{0}^{1,p(x)}(\Omega )\hookrightarrow L^{r(x)}(\Omega )%
\end{array}
\label{20}
\end{equation}%
is compact with $1\leq r(x)<p(x)^{\ast }$.

The next result gives important properties related to the logarithmic
nonlinearity.

\begin{lemma}
\label{alpha_beta_ineq}

\begin{enumerate}
\item[$(a)$] For each $\alpha ,\theta >0$, there is a constant $C$ that
depends only on $\alpha $ and $\theta $ such that 
\begin{equation*}
|\log (x)|\leq x^{-\alpha }+Cx^{\theta },
\end{equation*}%
for all $x> 0.$

\item[$(b)$] For each $\theta ,\epsilon >0$, there is a constant $C$ that
depends only on $\epsilon $ and $\theta $ such that 
\begin{equation*}
|log(x+\epsilon )|\leq x^{\theta }+C
\end{equation*}%
for all $x\geq 0.$

\item[$(c)$] Let $\gamma,\theta$ and $\delta$ be real numbers. If $%
\gamma,\theta >0$ and $\delta > \frac{\gamma}{\theta e}$ then the function $%
f(x)=\gamma x^{\delta }-\theta\log x,x>0$ attains a positive global minimum.
\end{enumerate}
\end{lemma}

\begin{proof}
With respect to the inequalities we only prove $(a)$ because $(b)$ can be
justified similarly. A simple computation provides $\lim_{x\rightarrow 0^{+}}%
\frac{|\log (x)|}{x^{-\alpha }}=0.$ Thus, there is a small $m>0$ such that 
\begin{equation*}
|\log (x)|\leq x^{-\alpha }\quad \text{for}\quad \quad x\in (0,m).
\end{equation*}%
On the other hand, the limit $\lim_{x\rightarrow +\infty }\frac{|\log (x)|}{%
x^{\theta }}=0$ implies that there is $M>0$ such that 
\begin{equation*}
|\log (x)|\leq x^{\theta }\quad \text{for}\quad x\in (M,+\infty ).
\end{equation*}%
Since the function $\frac{|\log (x)|}{x^{\theta }},$ $x>0$ is continuous for
all $x>0$, there is a constant, which depends on $\alpha $ and $\theta $,
such that $|\log (x)|\leq Cx^{\theta }$ in $[m,M].$ Therefore $|\log
(x)|\leq x^{-\alpha }+Cx^{\theta }$ for all $x>0,$ where the constant $C$
depends only on $\alpha $ and $\theta .$ \newline

In order to show $(c)$, observe that $f^{^{\prime }}(x)=\theta \delta
x^{\delta -1}-\gamma /x$. Then, $f$ has a unique critical point at $x_{0}=(%
\frac{\gamma }{\theta \delta })^{\frac{1}{\delta }}$. Thus, by solving the
inequations $f^{^{\prime }}(x)>0$ and $f^{^{\prime }}(x)<0$ for $x>0,$ it
follows that $f$ is increasing on the interval $[x_{0},+\infty )$ and
decreasing on $(-\infty ,x_{0}].$ By noticing that 
\begin{equation*}
f(x_{0})=\frac{\gamma }{\delta }\left( 1-\log \left( \frac{\gamma }{\theta
\delta }\right) \right) ,
\end{equation*}%
the condition $\delta >\frac{\gamma }{\theta e}$ implies that $f(x_{0})>0,$
which proves the result.
\end{proof}

\section{Sub-supersolution Theorem}

\label{S3}

Let us introduce the quasilinear system%
\begin{equation}
\left\{ 
\begin{array}{ll}
-\Delta _{p(x)}u=H(x,u,v) & \ \mbox{in}\ \Omega , \\ 
-\Delta _{q(x)}v=G(x,u,v) & \ \mbox{in}\ \Omega , \\ 
u=v=0 & \ \mbox{on}\ \partial \Omega ,%
\end{array}%
\right.  \label{pHG}
\end{equation}%
where $H,G:\Omega \times {\mathbb{R}}^{+}\times {\mathbb{R}}^{+}\rightarrow 
\mathbb{R}$ are Carath\'{e}odory functions satisfying the assumption:

\begin{description}
\item[\textrm{(}$\mathrm{H,G)}$] Given $T,S>0$, there is a constant $C>0$
such that%
\begin{equation*}
\text{ }|H(x,s,t)|,|G(x,s,t)|\leq C,\text{ for all \ }(x,s,t)\in \Omega
\times \lbrack 0,T]\times \lbrack 0,S].
\end{equation*}
\end{description}

The next result is a key point in the proof of Theorem \ref{T1}.

\begin{theorem}
\label{abs-theo} \label{abs-theo}Assume that $\mathrm{(H,G)}$ holds and
let $\underline{u}\in W_{0}^{1,p(x)}(\Omega )\cap L^{\infty }(\Omega )$ and $%
\underline{v}\in W_{0}^{1,q(x)}(\Omega )\cap L^{\infty }(\Omega )$ with $%
\underline{u},\underline{v}\geq 0$ in $\Omega $ and $\overline{u},\overline{v%
}\in W^{1,\infty }(\Omega )$ such that 
\begin{equation*}
\underline{u}\leq \overline{u}\ \ \text{and \ }\underline{v}\leq \overline{v}%
\text{ }\mbox{in}\ \Omega .
\end{equation*}%
Suppose that 
\begin{equation*}
\left\{ 
\begin{array}{l}
\int_{\Omega }|\nabla \underline{u}|^{p(x)-2}\nabla u\nabla \phi \ dx,\leq
\int_{\Omega }H(x,\underline{u},\underline{v})\phi \ dx, \\ 
\int_{\Omega }|\nabla \underline{v}|^{q(x)-2}\nabla v\nabla \psi \ dx,\leq
\int_{\Omega }G(x,\underline{u},\underline{v})\psi \ dx,%
\end{array}%
\right.
\end{equation*}%
and%
\begin{equation*}
\left\{ 
\begin{array}{l}
\int_{\Omega }|\nabla \overline{u}|^{p(x)-2}\nabla u\nabla \phi \ dx\geq
\int_{\Omega }H(x,\overline{u},\overline{v})\phi \ dx, \\ 
\int_{\Omega }|\nabla \overline{v}|^{p(x)-2}\nabla u\nabla \phi \ dx\geq
\int_{\Omega }H(x,\overline{u},\overline{v})\psi \ dx,%
\end{array}%
\right.
\end{equation*}%
for all nonnegative functions $(\phi ,\psi )\in W_{0}^{1,p(x)}(\Omega
)\times W_{0}^{1,q(x)}(\Omega )$. Then problem (\ref{pHG}) has a (positive)
solution $(u,v)\in (W_{0}^{1,p(x)}(\Omega )\cap L^{\infty }(\Omega ))\times
(W_{0}^{1,q(x)}(\Omega )\cap L^{\infty }(\Omega ))$ satisfying 
\begin{equation*}
\underline{u}(x)\leq u(x)\leq \overline{u}(x)\ \mbox{and}\ \underline{v}%
(x)\leq v(x)\leq \overline{v}(x),\text{ \ for a.e. }x\in \Omega .
\end{equation*}
\end{theorem}

\begin{proof}
The proof is chiefly based on pseudomonotone operator theory. Define the
functions 
\begin{equation*}
H_{1}(x,s,t)=\left\{ 
\begin{array}{ll}
H(x,\underline{u}(x),\underline{v}(x)), & s\leq \underline{u}(x) \\ 
H(x,s,\underline{v}(x)), & \underline{u}(x)\leq s\leq \overline{u}(x)\ %
\mbox{and}\ t\leq \underline{v}(x) \\ 
H(x,s,t), & \underline{u}(x)\leq s\leq \overline{u}(x)\ \mbox{and}\ 
\underline{v}(x)\leq t\leq \overline{v}(x) \\ 
H(x,s,\overline{v}(x)), & \underline{u}(x)\leq s\leq \overline{u}(x)\ %
\mbox{and}\ t\geq \overline{v}(x) \\ 
H(x,\overline{u}(x),\overline{v}(x)), & s\geq \overline{u}(x)\ \mbox{and}\
t\geq \overline{v}(x).%
\end{array}%
\right.
\end{equation*}%
and 
\begin{equation*}
G_{1}(x,s,t)=\left\{ 
\begin{array}{ll}
G(x,\underline{u}(x),\underline{v}(x)), & t\leq \underline{v}(x) \\ 
G(x,s,\underline{v}(x)), & \underline{v}(x)\leq t\leq \overline{v}(x)\ %
\mbox{and}\ s\leq \underline{u}(x) \\ 
G(x,s,t), & \underline{v}(x)\leq t\leq \overline{v}(x)\ \mbox{and}\ 
\underline{u}(x)\leq s\leq \overline{u}(x) \\ 
G(x,s,\overline{v}(x)), & \underline{v}(x)\leq t\leq \overline{v}(x)\ %
\mbox{and}\ s\geq \overline{u}(x) \\ 
G(x,\overline{u}(x),\overline{v}(x)), & t\geq \overline{v}(x).%
\end{array}%
\right.
\end{equation*}%
In what follows, we fix $l\in (0,1)$ with $\min \{p^{-},q^{-}\}>1+l$ and set 
\begin{equation*}
\gamma _{1}(x,s):=-((\underline{u}(x)-s)_{+})^{l}+((s-\overline{u}%
(x))_{+})^{l},
\end{equation*}%
\begin{equation*}
\gamma _{2}(x,s):=-((\underline{v}(x)-s)_{+})^{l}+((s-\overline{v}%
(x))_{+})^{l}.
\end{equation*}%
Using the above functions, let us introduce the auxiliary problem 
\begin{equation}
\left\{ 
\begin{array}{ll}
-\Delta _{p(x)}u=H_{2}(x,u,v) & \ \mbox{in}\ \Omega , \\ 
-\Delta _{q(x)}v=G_{2}(x,u,v) & \ \mbox{in}\ \Omega , \\ 
u=v=0 & \ \mbox{on}\ \partial \Omega ,%
\end{array}%
\right.  \label{auxiliary-system}
\end{equation}%
where 
\begin{equation}
H_{2}(x,s,t):=H_{1}(x,s,t)-\gamma _{1}(x,s)  \label{1}
\end{equation}%
and 
\begin{equation*}
G_{2}(x,s,t):=G_{1}(x,s,t)-\gamma _{2}(x,s).
\end{equation*}%
By Minty-Browder Theorem (see, e.g., \cite{Necas-book}), problem %
\eqref{auxiliary-system} has a solution $(u,v)$ in $W_{0}^{1,p(x)}(\Omega
)\times W_{0}^{1,q(x)}(\Omega )$. Indeed, let $B:E\rightarrow E^{\prime }$
be a function defined by%
\begin{equation*}
\begin{array}{l}
\langle B(u,v),(\phi ,\psi )\rangle :=\int_{\Omega }|\nabla
u|^{p(x)-2}\nabla u\nabla \phi +|\nabla v|^{q(x)-2}\nabla v\nabla \psi \ dx
\\ 
-\int_{\Omega }H_{2}(x,u,v)\phi \ dx-\int_{\Omega }G_{2}(x,u,v)\phi \ dx,%
\end{array}%
\end{equation*}%
where $E$ is the Banach space $W_{0}^{1,p(x)}(\Omega )\times
W_{0}^{1,q(x)}(\Omega )$ endowed with the norm 
\begin{equation*}
\Vert (u,v)\Vert =\max \{\Vert u\Vert _{1,p(x)},\Vert v\Vert _{1,q(x)}\},%
\text{ }(u,v)\in E.
\end{equation*}

Let us show that the function $B$ satisfies the hypotheses of Minty-Browder
Theorem. \newline

\noindent \textbf{i) B is continuous} \newline

Let $(u_{n},v_{n})\in E$ be a sequence that converges to $(u,v)$ in $E$. We
need to prove that $\Vert B(u_{n},v_{n})-B(u,v)\Vert _{E^{^{\prime
}}}\rightarrow 0.$ To this end, let $(\phi ,\psi )\in E$ with $\Vert (\phi
,\psi )\Vert _{E}\leq 1.$ By H\"{o}lder inequality, one has 
\begin{equation*}
\begin{array}{l}
\left\vert \int_{\Omega }|\nabla u_{n}|^{p(x)-2}\nabla u_{n}\nabla \phi
-|\nabla u|^{p(x)-2}\nabla u\nabla \phi \ dx\right\vert \\ 
\leq C\Vert |\nabla u_{n}|^{p(x)-2}\nabla u_{n}-|\nabla u|^{p(x)-2}\nabla
u\Vert _{L^{\frac{p(x)}{p(x)-1}}(\Omega )}.%
\end{array}%
\end{equation*}

Up to a subsequences, we can assume that $\nabla u_{n}(x)\rightarrow \nabla
u(x)$ a.e in $\Omega $ and there exists a function $U\in (L^{p(x)}(\Omega
))^{N}$ such that $|\nabla u_{n}(x)|\leq U(x)$ a.e in $\Omega .$ Therefore,
the Lebesgue's Dominated Convergence Theorem yields 
\begin{equation*}
\Vert |\nabla u_{n}|^{p(x)-2}\nabla u_{n}-|\nabla u|^{p(x)-2}\nabla u\Vert
_{L^{\frac{p(x)}{p(x)-1}}(\Omega )}\rightarrow 0.
\end{equation*}%
Note that 
\begin{equation*}
\begin{array}{l}
\left\vert \int_{\Omega }(H_{2}(x,u_{n},v_{n})-H_{2}(x,u,v))\phi \
dx\right\vert \\ 
\leq \int_{\Omega }|H_{1}(x,u_{n},v_{n})-H_{1}(x,u,v)||\phi |\
dx+\int_{\Omega }|\gamma _{1}(x,u_{n})-\gamma _{1}(x,u)||\phi |\ dx.%
\end{array}%
\end{equation*}%
Then, the continuity and the boundedness of $H$, together with Lebesgue's
Dominated convergence Theorem and H\"{o}lder inequality, gives 
\begin{equation*}
\sup_{\Vert \phi \Vert \leq 1}\int_{\Omega
}|H_{1}(x,u_{n},v_{n})-H_{1}(x,u,v)||\phi |\ dx\rightarrow 0.
\end{equation*}%
On the other hand, we can assume that $u_{n}(x)\rightarrow u(x)$ a.e in $%
\Omega $ and that exists $w\in L^{p(x)}(\Omega )$ such that $|u_{n}(x)|\leq
w(x)$ a.e in $\Omega .$ Arguing as before we get 
\begin{equation*}
\Vert \gamma _{1}(x,u)-\gamma _{1}(x,u_{n})\Vert _{L^{\frac{p(x)}{p(x)-1}%
(\Omega )}}\rightarrow 0,
\end{equation*}%
and so, 
\begin{equation*}
\sup_{\Vert \phi \Vert \leq 1}\int_{\Omega
}(H_{2}(x,u_{n},v_{n})-H_{2}(x,u,v))\phi \ dx\rightarrow 0.
\end{equation*}%
Hence, the previous reasoning provides 
\begin{equation*}
\Vert |\nabla v_{n}|^{q(x)-2}\nabla v_{n}-|\nabla v|^{q(x)-2}\nabla v\Vert
_{L^{\frac{q(x)}{q(x)-1}}(\Omega )}\rightarrow 0
\end{equation*}%
and 
\begin{equation*}
\sup_{\Vert \psi \Vert \leq 1}\int_{\Omega
}(G_{2}(x,u_{n},v_{n})-G_{2}(x,u,v))\psi \ dx\rightarrow 0,
\end{equation*}%
which justify the continuity of $B.$ \newline

\noindent \textbf{ii) B is bounded} \newline

Let us show that if $U\subset E$ is a bounded set then $B(U)\subset
E^{\prime }$ is bounded. To this end, consider a bounded set $U$ and $(\phi
,\psi )\in E$ such that $\Vert (\phi ,\psi )\Vert \leq 1$. Then, for $%
(u,v)\in U$ the H\"{o}lder inequality gives 
\begin{equation*}
\begin{array}{l}
\left\vert \int_{\Omega }|\nabla u|^{p(x)-2}\nabla u\nabla \phi +|\nabla
u|^{q(x)-2}\nabla v\nabla \psi \ dx\right\vert \\ 
\leq C(\Vert |\nabla u|^{p(x)-1}\Vert _{L^{\frac{p(x)}{p(x)-1}}(\Omega
)}+\Vert |\nabla v|^{q(x)-1}\Vert _{L^{\frac{q(x)}{q(x)-1}}(\Omega )})\leq C.%
\end{array}%
\end{equation*}%
Since $H_{1}(x,u,v)$ is bounded, we derive that 
\begin{equation*}
\left\vert \int_{\Omega }|H_{1}(x,u,v)|\phi |\ dx|\right\vert \leq
C\int_{\Omega }|\phi |\ dx\leq C.
\end{equation*}%
On the other hand, since 
\begin{equation*}
\int_{\Omega }|\gamma _{1}(x,u)|^{\frac{p(x)}{p(x)-1}}\ dx\leq C\int_{\Omega
}(1+|u(x)|+|\overline{u}(x)|+|\underline{u}(x)|)^{\frac{p(x)}{p(x)-1}}\ dx.
\end{equation*}%
the H\"{o}lder inequality ensures 
\begin{equation*}
\int_{\Omega }|\gamma _{1}(x,u)||\phi |\ dx\leq C.
\end{equation*}%
From the above arguments we obtain the boundedness of $B.$ \newline

\noindent \textbf{iii) B is coercive} \newline

Next, we prove that 
\begin{equation*}
\frac{\langle B(u,v),(u,v)\rangle }{\Vert (u,v)\Vert }\rightarrow +\infty \ 
\text{as}\ \Vert (u,v)\Vert \rightarrow +\infty .
\end{equation*}%
Note that 
\begin{equation}
\int_{\Omega }H_{1}(x,u,v)u\ dx\geq -\int_{\Omega }|H_{1}(x,u,v)||u|\ dx\geq
-C\Vert \nabla u\Vert _{L^{p(x)}(\Omega )},  \label{coer_1}
\end{equation}%
where $C$ is a positive constant. The triangular inequality and the fact
that $(a+b)^{\theta }\leq a^{\theta }+b^{\theta }$ for nonnegative numbers $%
a $ and $b$ with $\theta \in (0,1)$  give
\begin{equation*}
\begin{array}{l}
\int_{\Omega }-\gamma _{1}(x,u)u\ dx=\int_{\{\underline{u}\geq u\}}(%
\underline{u}-u)^{l}u\ dx-\int_{\{u\geq \overline{u}\}}(u-\overline{u}%
)^{l}u\ dx \\ 
\geq -\int_{\{\underline{u}\geq u\}}(|\underline{u}|+|u|)^{l}|u|\
dx-\int_{\{u\geq \overline{u}\}}(u-\overline{u})^{l}u\ dx \\ 
\geq -\int_{\{\underline{u}\geq u\}}(|\underline{u}|^{l}+|u|^{l})|u|\
dx-\int_{\{u\geq \overline{u}\}\cap \{u>0\}}(u-\overline{u})^{l}u\ dx \\ 
-\int_{\{u\geq \overline{u}\}\cap \{u<0\}}(u-\overline{u})^{l}u\ dx \\ 
\geq -\int_{\Omega }(|\underline{u}|^{l}+|u|^{l})|u|\ dx-\int_{\Omega }(|%
\overline{u}|^{l}+|u|^{l})|u|\ dx.%
\end{array}%
\end{equation*}%
Gathering the last inequality with the embeddings 
\begin{equation*}
W^{1,p(x)}(\Omega )\hookrightarrow L^{p(x)}(\Omega )\ \text{and}\
L^{p(x)}(\Omega )\hookrightarrow L^{1+l}(\Omega )
\end{equation*}%
we derive 
\begin{equation*}
\begin{array}{l}
\int_{\Omega }\gamma _{1}(x,u)u\ dx\geq -C\Vert u\Vert _{L^{p(x)}(\Omega
)}-\int_{\Omega }|u|^{1+l}\ dx \\ 
\geq -C\Vert \nabla u\Vert _{L^{p(x)}(\Omega )}-C\Vert \nabla u\Vert
_{L^{p(x)}(\Omega )}^{1+l}.%
\end{array}%
\end{equation*}%
From (\ref{1}), \eqref{coer_1} and the above inequality we have
\begin{equation*}
\begin{array}{l}
-\int_{\Omega }H_{2}(x,u,v)u\ dx\geq -C\Vert \nabla u\Vert _{L^{p(x)}(\Omega
)}-C\Vert \nabla u\Vert _{L^{p(x)}(\Omega )}^{1+l} \\ 
\geq -C\Vert (u,v)\Vert -C\Vert (u,v)\Vert ^{1+l},%
\end{array}%
\end{equation*}%
where $C$ is a positive constant. In the same manner, we can see that 
\begin{equation*}
\begin{array}{l}
-\int_{\Omega }G_{2}(x,u,v)u\ dx\geq -C\Vert \nabla v\Vert _{L^{q(x)}(\Omega
)}-C\Vert \nabla v\Vert _{L^{q(x)}(\Omega )}^{1+l} \\ 
\geq -C\Vert (u,v)\Vert -C\Vert (u,v)\Vert ^{1+l}.%
\end{array}%
\end{equation*}

\begin{itemize}
\item If $\Vert \nabla u\Vert _{L^{p(x)}(\Omega )}\geq 1$ and $\Vert \nabla
v\Vert _{L^{q(x)}(\Omega )}<1$,%
\begin{equation*}
\int_{\Omega }|\nabla u|^{p(x)}\ dx+\int_{\Omega }|\nabla v|^{q(x)}\ dx\geq
\Vert \nabla u\Vert _{L^{p(x)}(\Omega )}^{p_{-}}+\Vert \nabla v\Vert
_{L^{q(x)}(\Omega )}^{q_{+}}
\end{equation*}

\item If $\Vert \nabla u\Vert _{L^{p(x)}(\Omega )}\geq 1$ and $\Vert \nabla
v\Vert _{L^{q(x)}(\Omega )}\geq 1$, 
\begin{equation*}
\int_{\Omega }|\nabla u|^{p(x)}\ dx+\int_{\Omega }|\nabla v|^{q(x)}\ dx\geq
\Vert \nabla u\Vert _{L^{p(x)}(\Omega )}^{p_{-}}+\Vert \nabla u\Vert
_{L^{q(x)}(\Omega )}^{q_{-}}.
\end{equation*}
\end{itemize}

Consider in $E$ a sequence $\{(u_{n},v_{n})\}_{n}$ such that $\Vert
(u_{n},v_{n})\Vert \rightarrow +\infty .$ Thus $\Vert \nabla u_{n}\Vert
_{L^{p(x)}(\Omega )}\rightarrow +\infty $ or $\Vert \nabla v_{n}\Vert
_{L^{q(x)}(\Omega )}\rightarrow +\infty $. Suppose that the first
possibility happens and that $\Vert \nabla u_{n}\Vert _{L^{p(x)}(\Omega
)}\geq 1$ for all $n\in \mathbb{N}.$ Then, we consider two cases:

\begin{itemize}
\item $\Vert \nabla u_{n}\Vert _{L^{p(x)}(\Omega )}\geq 1$ and $\Vert \nabla
v_{n}\Vert _{L^{q(x)}(\Omega )}<1$ for $n\in \mathbb{N}.$ In this case we
have 
\begin{equation*}
\begin{array}{l}
\frac{\langle B(u_{n},v_{n}),(u_{n},v_{n})\rangle }{\Vert (u_{n},v_{n})\Vert
_{E}}\geq \frac{\Vert \nabla u_{n}\Vert _{L^{p(x)}(\Omega )}^{p_{-}}+\Vert
\nabla v_{n}\Vert _{L^{q(x)}(\Omega )}^{q_{+}}}{\Vert \nabla u_{n}\Vert
_{L^{p(x)}(\Omega )}}-C\frac{\Vert \nabla v_{n}\Vert _{L^{q(x)}(\Omega
)}^{1+l}}{\Vert \nabla u_{n}\Vert _{L^{p(x)}(\Omega )}}-C\\
-C\Vert \nabla
u_{n}\Vert _{L^{p(x)}(\Omega )}^{l} \\ 
\geq \Vert \nabla u_{n}\Vert _{L^{p(x)}(\Omega )}^{p^{-}-1}-\frac{C}{\Vert
\nabla u_{n}\Vert _{L^{p(x)}(\Omega )}}-C-\Vert \nabla u_{n}\Vert
_{L^{p(x)}(\Omega )}^{l}.%
\end{array}%
\end{equation*}

\item $\Vert \nabla u_{n}\Vert _{L^{p(x)}(\Omega )}\geq 1$ and $\Vert \nabla
v_{n}\Vert _{L^{q(x)}(\Omega )}\geq 1$ for $n\in \mathbb{N}.$ In this second
case, we have 
\begin{equation*}
\begin{array}{l}
\frac{\langle B(u_{n},v_{n}),(u_{n},v_{n})\rangle }{\Vert (u_{n},v_{n})\Vert
_{E}}\geq \frac{(\max \{\Vert \nabla u_{n}\Vert _{L^{p(x)}(\Omega )},\Vert
\nabla v_{n}\Vert _{L^{q(x)}(\Omega )}\})^{\min \{p^{-},q^{-}\}}}{\max
\{\Vert \nabla u_{n}\Vert _{L^{p(x)}(\Omega )},\Vert \nabla v_{n}\Vert
_{L^{q(x)}(\Omega )}\}} \\ 
-C-C(\max \{\Vert \nabla u_{n}\Vert _{L^{p(x)}(\Omega )},\Vert \nabla
v_{n}\Vert _{L^{q(x)}(\Omega )}\})^{l}.%
\end{array}%
\end{equation*}
\end{itemize}

Consequently, in both cases studied above, one has%
\begin{equation*}
\frac{\langle B(u_{n},v_{n}),(u_{n},v_{n})\rangle }{\Vert (u_{n},v_{n})\Vert
_{E}}\ \text{as}\ n\rightarrow +\infty .
\end{equation*}%
The other situations regarding $\Vert \nabla u_{n}\Vert _{L^{p(x)}(\Omega )}$
and $\Vert \nabla v_{n}\Vert _{L^{q(x)}(\Omega )}$ can be handled in much
the same way.\newline

\noindent \textbf{iv) B is pseudomonotone} \newline

We recall that $B$ is a pseudomonotone operator if $(u_{n},v_{n})%
\rightharpoonup (u,v)$ in $E$ and 
\begin{equation}
\limsup_{n\rightarrow +\infty }\langle
B(u_{n},v_{n}),(u_{n},v_{n})-(u,v)\rangle \leq 0,  \label{lim-sup-cond}
\end{equation}%
then 
\begin{equation*}
\liminf_{n\rightarrow +\infty }\langle B(u_{n},v_{n}),(u_{n},v_{n})-(\phi
,\psi )\rangle \geq \langle B(u,v),(u,v)-(\phi ,\psi )\rangle ,
\end{equation*}%
for all $(\phi ,\psi )\in E.$

If $(u_{n},v_{n})\rightharpoonup (u,v)$ then $u_{n}\rightharpoonup u$ and $%
v_{n}\rightharpoonup v$ in $W^{1,p(x)}(\Omega )$ and $W^{1,q(x)}(\Omega ),$
respectively. Since $H_{1}$ and $G_{1}$ are bounded we must have 
\begin{equation*}
\int_{\Omega }H_{1}(x,u_{n},v_{n})(u_{n}-u)\ dx\rightarrow 0
\end{equation*}%
and 
\begin{equation*}
\int_{\Omega }G_{1}(x,u_{n},v_{n})(u_{n}-u)\ dx\rightarrow 0.
\end{equation*}%
Note that%
\begin{equation*}
\begin{array}{l}
\langle B(u_{n},v_{n}),(u_{n},v_{n})-(u,v)\rangle =\int_{\Omega }\langle
|\nabla u_{n}|^{p(x)-2}\nabla u_{n},\nabla u_{n}-\nabla u\rangle \ dx \\ 
+\int_{\Omega }H_{2}(x,u_{n},v_{n})(u_{n}-u)\ dx+\int_{\Omega }\langle
|\nabla v_{n}|^{q(x)-2}\nabla v_{n},\nabla v_{n}-\nabla v\rangle \ dx \\ 
+\int_{\Omega }G_{2}(x,u_{n},v_{n})(v_{n}-v)\ dx.%
\end{array}%
\end{equation*}%
The previous arguments can be repeated to show that 
\begin{equation*}
\int_{\Omega }H_{2}(x,u_{n},v_{n})(u_{n}-u)\ dx=\int_{\Omega
}[H_{1}(x,u_{n},v_{n})(u_{n}-u)-\gamma _{1}(x,u_{n})(u_{n}-u)]\ dx.
\end{equation*}%
\begin{equation*}
\lim_{n\rightarrow +\infty }\int_{\Omega }H_{2}(x,u_{n},v_{n})(u_{n}-u)\
dx=0.
\end{equation*}%
and 
\begin{equation*}
\lim_{n\rightarrow +\infty }\int_{\Omega }G_{2}(x,u_{n},v_{n})(v_{n}-v)\
dx=0.
\end{equation*}%
Gathering the above limits together with \eqref{lim-sup-cond}, one has 
\begin{equation}
\begin{array}{l}
\overline{\lim} \int_{\Omega }\langle |\nabla
u_{n}|^{p(x)-2}\nabla u_{n},\nabla (u_{n}-u)\rangle \ dx+\int_{\Omega
}\langle |\nabla v_{n}|^{q(x)-2}\nabla v_{n},\nabla (v_{n}-v)\ dx
\\
\leq 0.%
\end{array}
\label{lim-sup-ineq}
\end{equation}%
From the weak convergence, we get 

$$\int_{\Omega }\langle |\nabla u|^{p(x)-2}\nabla u,\nabla (u_{n}-u)\rangle
dx=o_{n}(1)$$
and
$$\int_{\Omega }\langle |\nabla v|^{p(x)-2}\nabla
v,\nabla (v_{n}-v)\rangle dx=o_{n}(1).$$

Therefore 
\begin{equation}
\begin{array}{l}
\int_{\Omega }\langle |\nabla u_{n}|^{p(x)-2}\nabla u_{n},\nabla
(u_{n}-u)\rangle \ dx \\ 
=\int_{\Omega }\langle |\nabla u_{n}|^{p(x)-2}\nabla u_{n}-|\nabla
u|^{p(x)-2}\nabla u,\nabla (u_{n}-u)\rangle \ dx+o_{n}(1)%
\end{array}
\label{2}
\end{equation}%
and 
\begin{equation}
\begin{array}{l}
\int_{\Omega }\langle |\nabla v_{n}|^{p(x)-2}\nabla v_{n},\nabla
(v_{n}-v)\rangle \ dx \\ 
=\int_{\Omega }\langle |\nabla v_{n}|^{p(x)-2}\nabla v_{n}-|\nabla
v|^{p(x)-2}\nabla v,\nabla (v_{n}-v)\rangle \ dx+o_{n}(1).%
\end{array}
\label{3}
\end{equation}%
Using (\ref{2}) and (\ref{3}) in \eqref{lim-sup-ineq}, the $(S^{+})$
property of the operators $-\Delta _{p(x)}$ and $-\Delta _{q(x)}$ garantees
that $u_{n}\rightarrow u$ in $W_{0}^{1,p(x)}(\Omega )$ and $v_{n}\rightarrow
v$ in $W_{0}^{1,q(x)}(\Omega )$. Thus, by continuity of $B$, it turn out that%
\begin{equation*}
\lim_{n\rightarrow +\infty }\langle B(u_{n},v_{n}),(u_{n},v_{n})-(\phi ,\psi
)\rangle =\langle B(u,v),(u,v)-(\phi ,\psi ),
\end{equation*}%
for all $(\phi ,\psi )\in E.$

Finally, from properties I)- IV) we are in a position to apply \cite[Theorem
3.3.6]{Necas-book} which ensures that $B$ is surjective. Thereby, there
exists $(u,v)\in E$ such that 
\begin{equation*}
\langle B(u,v),(\phi ,\psi )\rangle =0,\forall (\phi ,\psi )\in E,
\end{equation*}%
and in particular, $(u,v)$ is a solution of \eqref{auxiliary-system}.

It remains to prove that%
\begin{equation}
\underline{u}\leq u\leq \ \overline{u}\quad \text{and}\quad \underline{v}%
\leq v\leq \overline{v}\quad \mbox{in}\quad \Omega .  \label{4}
\end{equation}

We only prove the first inequalities in (\ref{4}) because the second ones
can be justified similarly. Set $(\phi ,\psi ):=((u-\overline{u})_{+},0)$.
From the definition of $H_{2}$, we obtain 
\begin{equation*}
\begin{array}{l}
\int_{\Omega }|\nabla u|^{p(x)-2}\nabla u\nabla (u-\overline{u}%
)_{+}=\int_{\{u\geq \overline{u}\}}H_{1}(x,u,v)(u-\overline{u})_{+}\ dx \\ 
\text{ \ \ \ \ \ \ }-\int_{\{u\geq \overline{u}\}}(-((\underline{u}%
-u)_{+})^{l}+((u-\overline{u})_{+})^{l})(u-\overline{u})_{+}\ dx \\ 
=\int_{\{u\geq \overline{u}\}}H(x,\overline{u},\overline{v})(u-\overline{u}%
)_{+}\ dx-\int_{\Omega }((u-\overline{u})_{+})^{l+1}\ dx \\ 
\leq \int_{\Omega }|\nabla \overline{u}|^{p(x)-2}\nabla \overline{u}\nabla
(u-\overline{u})_{+}\ dx-\int_{\Omega }((u-\overline{u})_{+})^{l+1}\ dx.%
\end{array}%
\end{equation*}%
Therefore 
\begin{equation*}
\int_{\Omega }\langle |\nabla u|^{p(x)-2}\nabla u,|\nabla \overline{u}%
|^{p(x)-2}\nabla \overline{u},\nabla {(u-\overline{u})_{+}}\rangle \ dx\leq
-\int_{\Omega }((u-\overline{u})_{+})^{l+1}\ dx\leq 0
\end{equation*}%
wich implies that $u\leq \overline{u}$ in $\Omega .$ Using a quite similar
argument for $(\phi ,\psi ):=((\underline{u}-u)_{+},0)$ we get $\underline{u}%
\leq u$ in $\Omega .$ This completes the proof.
\end{proof}

\section{Proof of Theorem \protect\ref{T1}}

\label{S4}

For every $\varepsilon >0$, let us introduce the auxiliary problem%
\begin{equation}
\left\{ 
\begin{array}{ll}
-\Delta _{p(x)}u=-\gamma \log (|v|+\epsilon )+\theta |v|^{\alpha (x)} & \ 
\text{in}\ \Omega , \\ 
-\Delta _{q(x)}v=-\gamma \log (|u|+\epsilon )+\theta |u|^{\beta (x)} & \text{
in}\ \Omega , \\ 
u=v=0 & \text{ on }\partial \Omega .%
\end{array}%
\right.  \label{System-log-perturbate}
\end{equation}%
Our goal is to show through Theorem \ref{abs-theo} that (\ref%
{System-log-perturbate}) has a positive solution $(u_{\epsilon },v_{\epsilon
})$. Then, by passing to the limit as $\epsilon \rightarrow 0^{+}$ we get a
solution for the original problem (\ref{p}).\bigskip

Let $\tilde{\Omega}$ be a bounded domain in $%
\mathbb{R}
^{N}$ with smooth boundary $\partial \tilde{\Omega}$ such that $\overline{%
\Omega }\subset \tilde{\Omega}$ and denote by $\tilde{d}(x)=dist(x,\partial 
\tilde{\Omega}).$ In \cite[Lemma 3.1]{YY}, the authors have proved that, for 
$\delta >0$ small enough and for constants $\eta >0$, the function $w\in
C^{1}(\overline{\tilde{\Omega}})\cap C_{0}(\tilde{\Omega})$ defined by%
\begin{equation*}
w(x)=\left\{ 
\begin{array}{ll}
\xi \tilde{d}(x) & \text{if }\tilde{d}(x)<\delta \\ 
\xi \delta +\int_{\delta }^{\tilde{d}(x)}\xi \left( \frac{2\delta -t}{\delta 
}\right) ^{\frac{2}{p^{-}-1}} & \text{if }\delta \leq \tilde{d}(x)< 2\delta
\\ 
\xi \delta +\int_{\delta }^{2\delta }\xi \left( \frac{2\delta -t}{\delta }%
\right) ^{\frac{2}{p^{-}-1}} & \text{if }2\delta \leq \tilde{d}(x) ,%
\end{array}%
\right.
\end{equation*}%
is a subsolution of the problem 
\begin{equation*}
\left\{ 
\begin{array}{ll}
-\Delta _{p(x)}u=\eta & \text{ in }\tilde{\Omega}, \\ 
u=0 & \text{ on }\partial \tilde{\Omega},%
\end{array}%
\right.
\end{equation*}%
where $\delta >0$ is a number that does not depend on $\eta$ and $\xi =c_{0}\eta ^{\frac{1}{p^{+}-1+\tau }}$ with  $\tau \in (0,1)$ a fixed number and $c_{0}>0$ is a number depending only on $\delta,\tau,\tilde{\Omega}$ and $p$. Note that
\begin{equation}
\left\{ 
\begin{array}{ll}
w(x)=c_{0}\eta ^{\frac{1}{p^{+}-1+\tau }}\tilde{d}(x) & \text{for }\tilde{d}%
(x)<\delta \\ 
c_{0}\eta ^{\frac{1}{p^{+}-1+\tau }}\delta \leq w(x) & \text{for }\tilde{d}%
(x)\geq \delta.%
\end{array}%
\right.  \label{w}
\end{equation}

Given $\lambda >1$, let $\overline{u}$ and $\overline{v}$ in $C^{1}(%
\overline{\tilde{\Omega}})$ be the unique solutions of problems 
\begin{equation}
\left\{ 
\begin{array}{ll}
-\Delta _{p(x)}\overline{u}=\lambda ^{\sigma } & \ in\ \tilde{\Omega} \\ 
\overline{u}=0 & \ on\ \partial \tilde{\Omega}%
\end{array}%
\right. ,\text{ \ }\left\{ 
\begin{array}{ll}
-\Delta _{q(x)}\overline{v}=\lambda ^{\sigma } & \ in\ \tilde{\Omega} \\ 
\overline{v}=0 & \ on\ \partial \tilde{\Omega}.%
\end{array}%
\right.  \label{super-first}
\end{equation}%
where $\sigma $ is a real constant.

If $\sigma >0$, considering the corresponding function $w$ for $\eta
=\lambda ^{\sigma }$ and applying the weak maximum principle we get 
\begin{equation}
\left\{ 
\begin{array}{c}
C_{0}\lambda ^{\frac{\sigma }{p^{+}-1+\tau _{1}}}\min \{\delta ,\tilde{d}%
(x)\}\leq \overline{u}(x)\leq C_{1}\lambda ^{\frac{\sigma }{p^{-}-1}} \\ 
C_{0}^{\prime }\lambda ^{\frac{\sigma }{q^{+}-1+\tau _{2}}}\min \{\delta ,%
\tilde{d}(x)\}\leq \overline{v}(x)\leq C_{1}^{\prime }\lambda ^{\frac{\sigma 
}{q^{-}-1}}%
\end{array}%
\right. \text{ in }\tilde{\Omega},  \label{below-above-sigma}
\end{equation}%
where $C_{0},C_{0}^{\prime },C_{1},C_{1}^{\prime }>0$ and $\tau _{1},\tau
_{2}\in (0,1)$ are constants that does not depend on $\lambda .$ If $%
-1<\sigma <0, $ from \cite[Lemma 2.1]{FZZ} and for $\lambda $ large, one has%
\begin{equation}
\overline{u}(x)\leq k_{2}\lambda ^{\frac{\sigma }{p^{-}-1}}\leq c_{2}\lambda
^{\frac{\sigma }{p^{+}-1}}\ \text{and}\ \overline{v}(x)\leq k_{2}^{\prime
}\lambda ^{\frac{\sigma }{q^{-}-1}}\leq c_{2}^{\prime }\lambda ^{\frac{%
\sigma }{q^{+}-1}}\ \text{in}\ \tilde{\Omega}  \label{2-above}
\end{equation}%
where $k_{2},k_{2}^{\prime },c_{2}$ and $c_{2}^{\prime }$ are positive
constants independent of $\lambda $ . Moreover, by the strong maximum
principle there is a constant $c_{0}>0$ (that can depend on $\lambda $) such
that 
\begin{equation}
c_{0}\tilde{d}(x)\leq \min \{\overline{u}(x),\overline{v}(x)\}\ 
\label{strong-principle}
\end{equation}

Now, let $\underline{u}$ and $\underline{v}$ in $C^{1}(\overline{\Omega })$
be the unique solutions of the homogeneous Dirichlet problems%
\begin{equation}
\left\{ 
\begin{array}{ll}
-\Delta _{p(x)}\underline{u}=\lambda ^{-1} & \ in\ \Omega , \\ 
\underline{u}=0 & \ on\ \partial \Omega .%
\end{array}%
\right. ,\text{ }\left\{ 
\begin{array}{ll}
-\Delta _{q(x)}\underline{v}=\lambda ^{-1} & \ in\ \Omega , \\ 
\underline{v}=0 & \ on\ \partial \Omega .%
\end{array}%
\right.  \label{10}
\end{equation}
By \cite[Lemma 2.1]{Fan1} and \cite{FZZ}, there exist positive constants $%
k_{0},K_{1}$ and $K_{2}$, independent of $\lambda $, such that 
\begin{equation}
\begin{array}{l}
\underline{u}(x)\leq K_{1}\lambda ^{\frac{-1}{p^{-}-1}}\text{ \ and \ }%
\underline{v}(x)\leq K_{2}\lambda ^{\frac{-1}{q^{-}-1}}\text{ \ in }{\Omega }%
\end{array}
\label{7.1}
\end{equation}%
and%
\begin{equation}
\begin{array}{c}
k_{0}d(x)\leq \min \{\underline{u}(x),\underline{v}(x)\}\text{ \ in }\Omega 
\text{.}%
\end{array}
\label{7.2}
\end{equation}%
By the weak maximum principle we have, $\underline{u}\leq \overline{u}$ and $%
\underline{v}\leq \overline{v}$ in $\overline{\Omega }$ for $\lambda >1$
sufficiently large.

We state the following existence result for the regularized problem (\ref%
{System-log-perturbate}).

\begin{theorem}
\label{perturbed-solutions} Under assumptions of Theorem \ref{T1}, there
exists $\varepsilon _{0}>0$ such that system (\ref{System-log-perturbate})
has a positive solution $(u_{\varepsilon },v_{\varepsilon })\in
(W_{0}^{1,p(x)}(\Omega )\cap L^{\infty }(\Omega ))\times
(W_{0}^{1,q(x)}(\Omega )\cap L^{\infty }(\Omega ))$, for all $\varepsilon
\in (0,\varepsilon _{0})$. Moreover, it hold 
\begin{equation}
\underline{u}(x)\leq u_{\varepsilon }(x)\leq \overline{u}(x)\ \ \text{and}\
\ \underline{v}(x)\leq v_{\varepsilon }(x)\leq \overline{v}(x)\ \ \text{for
a.e.}\ x\in \Omega .  \label{super-sub-ineq}
\end{equation}
\end{theorem}

\begin{proof}
First, let us show that $(\underline{u},\underline{v})$ is a subsolution for
problem (\ref{System-log-perturbate}) for all $\epsilon \in (0,\epsilon
_{0}) $. To this end, pick $\varepsilon _{0}\leq 1/2$. Then, from (\ref{10})
and (\ref{7.1}), for all $\epsilon \in (0,\varepsilon _{0}),$ one has%
\begin{equation*}
\begin{array}{l}
-\Delta _{p(x)}\underline{u}={\lambda }^{-1}\leq -\gamma \log (K_{2}\lambda
^{\frac{-1}{q^{-}-1}}+\varepsilon _{0})\leq -\gamma \log (\underline{v}%
(x)+\epsilon ) \\ 
\leq -\gamma \log (\underline{v}(x)+\epsilon )+\theta \underline{v}%
(x)^{\alpha (x)}\text{ \ in }\Omega%
\end{array}%
\end{equation*}%
and 
\begin{equation*}
\begin{array}{l}
-\Delta _{q(x)}\underline{v}=\lambda ^{-1}\leq -\gamma \log (K_{1}\lambda ^{%
\frac{-1}{p^{-}-1}}+\varepsilon _{0})\leq -\gamma \log (\underline{u}%
(x)+\epsilon ) \\ 
\leq -\gamma \log (\underline{u}(x)+\epsilon )+\theta \underline{u}%
(x)^{\beta (x)}\text{ \ in }\Omega ,%
\end{array}%
\end{equation*}%
for all $\gamma ,\theta >0$, provided that $\lambda >1$ is sufficiently
large.

Next, we will show that $(\overline{u},\overline{v})$ is a supersolution for
problem (\ref{System-log-perturbate}) for all $\epsilon \in (0,\epsilon
_{0}) $. Denote by $\delta ^{\prime }:=dist(\partial \tilde{\Omega},\partial
\Omega )$ and fix $\varepsilon _{0}\in (0,1).$ By Lemma \ref{alpha_beta_ineq}%
, there are constants $\sigma _{1},\sigma _{2}\in (0,1)$ and $C_{\sigma
_{1},\alpha ^{+}},C_{\sigma _{2,\beta ^{+}}}>0$ such that, for all $\epsilon
\in (0,\epsilon _{0})$, one has%
\begin{equation}
\begin{array}{l}
-\gamma \log (\overline{v}+\epsilon )+\theta \overline{v}^{\alpha (x)}\leq
\gamma (\frac{1}{(\overline{v}+\epsilon )^{\sigma _{1}}}+C_{\sigma
_{1},\alpha ^{+}}(\overline{v}+\epsilon )^{\alpha ^{+}})+\theta \overline{v}%
^{\alpha (x)}%
\end{array}
\label{log1}
\end{equation}%
and%
\begin{equation}
\begin{array}{l}
-\gamma \log (\overline{u}+\epsilon )+\theta \overline{u}^{\beta (x)}\leq
\gamma (\frac{1}{(\overline{u}+\epsilon )^{\sigma _{2}}}+C_{\sigma
_{2},\beta ^{+}}(\overline{u}+\epsilon )^{\beta ^{+}})+\theta \overline{u}%
^{\beta (x)}.%
\end{array}
\label{log2}
\end{equation}

If (\ref{H2}) holds, it follows from (\ref{below-above-sigma}), (\ref{log1}%
), (\ref{log2}) and for $\sigma >0$ in (\ref{super-first}), that 
\begin{equation}
\begin{array}{l}
-\gamma \log (\overline{v}+\epsilon )+\theta \overline{v}^{\alpha (x)}\leq
\gamma (\frac{1}{\overline{v}^{\sigma _{1}}}+C_{\sigma _{1},\alpha ^{+}}(%
\overline{v}+\epsilon _{0})^{\alpha ^{+}})+\theta (\overline{v}+1)^{\alpha
^{+}} \\ 
\leq \frac{\gamma }{\overline{v}^{\sigma _{1}}}+2^{\alpha ^{+}-1}(\gamma
C_{\sigma _{1},\alpha ^{+}}+\theta )(\overline{v}^{\alpha ^{+}}+1) \\ 
\leq \frac{\gamma }{(\overline{v})^{\sigma _{1}}}+2^{\alpha ^{+}-1}(\gamma
C_{\sigma _{1},\alpha ^{+}}+\theta )C_{1}^{\prime }(\lambda ^{\frac{\sigma
\alpha ^{+}}{q^{-}-1}}+1) \\ 
\frac{\gamma }{(C_{0}^{\prime }\lambda ^{\frac{\sigma }{q^{+}-1+\tau _{2}}%
}\min \{\delta ,\delta ^{^{\prime }}\})^{\sigma _{1}}}+2^{\alpha
^{+}-1}(\gamma C_{\sigma _{1},\alpha ^{+}}+\theta )(C_{1}^{\prime }\lambda ^{%
\frac{\sigma \alpha ^{+}}{q^{-}-1}}+1)\leq \lambda ^{\sigma }\text{ \ in }%
\Omega ,%
\end{array}
\label{30}
\end{equation}%
and%
\begin{equation}
\begin{array}{l}
-\gamma \log (\overline{u}+\epsilon )+\theta \overline{u}^{\beta (x)}\leq
\gamma (\frac{1}{\overline{u}^{\sigma _{2}}}+C_{\sigma _{2,\beta ^{+}}}(%
\overline{u}+\epsilon )^{\beta ^{+}})+\theta (\overline{u}+1)^{\beta ^{+}}
\\ 
\leq \frac{\gamma }{\overline{u}^{\sigma _{2}}}+2^{\beta ^{+}-1}(\gamma
C_{\sigma _{2},\beta ^{+}}+\theta )(\overline{u}^{\beta ^{+}}+1) \\ 
\leq \frac{\gamma }{(C_{0}\lambda ^{\frac{\sigma }{p^{+}-1+\tau _{1}}}\min
\{\delta ,\delta ^{\prime }\})^{\sigma _{2}}}+2^{\beta ^{+}}(\gamma
C_{\sigma _{2},\alpha ^{+}}+\theta )(C_{1}\lambda ^{\frac{\sigma \beta ^{+}}{%
p^{-}-1}}+1)\leq \lambda ^{\sigma }\text{ \ in }\Omega ,%
\end{array}
\label{30*}
\end{equation}%
for all $\gamma ,\theta >0$, provided that $\lambda $ is large enough.

If (\ref{hip-superlinear}) is satisfied, combining Lemma 1 with (\ref%
{2-above}) and (\ref{strong-principle}), by (\ref{log1}), (\ref{log2}) and
for $\sigma \in (-1,0)$ in (\ref{super-first}), we get%
\begin{equation}
\begin{array}{l}
-\gamma \log (\overline{v}+\epsilon )+\theta \overline{v}^{\alpha (x)}\leq
\gamma (\frac{1}{(\overline{v}+\epsilon )^{\sigma _{1}}}+C_{\sigma
_{1},\alpha ^{-}}(\overline{v}+\epsilon )^{\alpha ^{-}})+\theta \overline{v}%
^{\alpha (x)} \\ 
\leq \gamma (\frac{1}{\overline{v}^{\sigma _{1}}}+C_{\sigma _{1},\alpha
^{-}}2^{\alpha ^{-}-1}{\overline{v}}^{\alpha ^{-}}+C_{\sigma _{1},\alpha
^{-}}\varepsilon _{0}^{\alpha ^{-}}2^{\alpha ^{-}-1})+\theta {\overline{v}}%
^{\alpha (x)} \\ 
\leq \gamma (\frac{1}{(c_{0}\delta ^{\prime })^{\sigma _{1}}}+C_{\sigma
_{1},\alpha ^{-}}2^{\alpha ^{-}-1}{\lambda }^{\frac{\sigma \alpha ^{-}}{%
q^{+}-1}}+C_{\sigma _{1},\alpha ^{-}}2^{\alpha ^{-}-1})+\theta c_{2}^{\prime
}{\lambda }^{\frac{\sigma \alpha (x)}{q^{+}-1}} \\ 
\leq \gamma (\frac{1}{(c_{0}\delta ^{\prime })^{\sigma _{1}}}+C_{\sigma
_{1},\alpha ^{-}}2^{\alpha ^{-}-1}{\lambda }^{\frac{\sigma \alpha ^{-}}{%
q^{+}-1}}+C_{\sigma _{1},\alpha ^{-}}2^{\alpha ^{-}-1})+\theta c_{2}^{\prime
}{\lambda }^{\frac{\sigma \alpha ^{-}}{q^{+}-1}}\leq \lambda ^{\sigma }\text{
\ in }\Omega%
\end{array}
\label{31}
\end{equation}%
and 
\begin{equation}
\begin{aligned} -\gamma \log (\overline{u}+\epsilon )+\overline{u}^{\beta
(x)} &\leq \frac{\gamma }{(c_{0}\delta^{\prime} )^{\sigma _{1}}}+\gamma
C_{\sigma_{1},\beta ^{-}}2^{\beta^{-}-1}
{\lambda}^{\frac{\sigma\alpha^{-}}{q^{+}-1}}+\gamma C_{\sigma_1,
\beta^{-}}2^{\beta^{-}-1} \\
&+c^{\prime}_{2}{\lambda}^{\frac{\sigma\beta^{-}}{p^{+}-1}} \\ & \leq \lambda
^{\sigma }\text{ \ in }\Omega .\end{aligned}  \label{31*}
\end{equation}%
for $\gamma >0$ small enough, for all $\theta >0$ and all $\varepsilon \in
(0,\varepsilon _{0})$, provided that $\lambda $ is sufficiently large.

Finally, if (\ref{c}) holds, using (\ref{below-above-sigma}), (\ref%
{strong-principle}), (\ref{log1}) and (\ref{log2}), for $\sigma \in (-1,0)$
in (\ref{super-first}), we obtain%
\begin{equation}
\begin{array}{l}
-\gamma \log (\overline{v}+\epsilon )+\theta \overline{v}^{\alpha (x)}\leq
\gamma (\frac{1}{(\overline{v}+\epsilon )^{\sigma _{1}}}+C_{\sigma
_{1},\alpha ^{+}}(\overline{v}+1)^{\alpha ^{+}})+\theta \overline{v}^{\alpha
(x)} \\ 
\leq \gamma (\frac{1}{\overline{v}^{\sigma _{1}}}+C_{\sigma _{1},\alpha
^{+}}(\overline{v}^{\alpha ^{+}}+1))+\theta (\overline{v}+1)^{\alpha ^{+}}
\\ 
\leq \gamma (\frac{1}{(c_{0}\delta ^{\prime })^{\sigma _{1}}}+C_{\sigma
_{1},\alpha ^{+}}(\overline{v}^{\alpha ^{+}}+1))+2^{\alpha ^{+}}\theta (%
\overline{v}^{\alpha ^{+}}+1) \\ 
\leq \gamma (\frac{1}{(c_{0}\delta ^{\prime })^{\sigma _{1}}}+C_{\sigma
_{1},\alpha ^{+}}(\lambda ^{\frac{\sigma \alpha ^{+}}{q^{-}-1}%
}+1))+2^{\alpha ^{+}}\theta (\lambda ^{\frac{\sigma \alpha ^{+}}{q^{-}-1}%
}+1)\leq \lambda ^{\sigma }\text{ \ in }\overline{\Omega },%
\end{array}
\label{32}
\end{equation}
and similarly 
\begin{equation}
\begin{array}{c}
-\gamma \log (\overline{u}+\epsilon )+\theta \overline{u}^{\beta (x)}\leq
\gamma (\frac{1}{\overline{u}^{\sigma _{2}}}+C_{\sigma _{2},\beta ^{+}}(%
\overline{u}+1)^{\beta ^{+}})+\theta \overline{u}^{\beta (x)}\leq \lambda
^{\sigma }\text{ \ in }\overline{\Omega },%
\end{array}
\label{32*}
\end{equation}%
for all $\gamma ,\theta >0$ small and all $\varepsilon \in (0,\varepsilon
_{0})$, provided that $\lambda >0$ is large enough.

Consequently, it turns out from (\ref{30}), (\ref{30*}), (\ref{31}), (\ref%
{31*}), (\ref{32}) and (\ref{32*}) that 
\begin{equation*}
\begin{array}{c}
\int_{\Omega }\left\vert \nabla \overline{u}\right\vert ^{p(x)-2}\nabla 
\overline{u}\nabla \varphi \text{ }dx\geq \int_{\Omega }(-\gamma \log (%
\overline{v}+\epsilon )+\theta \overline{v}^{\alpha (x)})\varphi \text{ }dx%
\end{array}%
\end{equation*}%
and%
\begin{equation*}
\begin{array}{c}
\int_{\Omega }\left\vert \nabla \overline{v}\right\vert ^{q(x)-2}\nabla 
\overline{v}\nabla \psi \text{ }dx\geq \int_{\Omega }(-\gamma \log (%
\overline{u}+\epsilon )+\theta \overline{u}^{\alpha (x)})\psi \text{ }dx,%
\end{array}%
\end{equation*}%
for all $\left( \varphi ,\psi \right) \in W_{0}^{1,p(x)}\left( \Omega
\right) \times W_{0}^{1,q(x)}\left( \Omega \right) $ with $\varphi ,\psi
\geq 0$. This shows that $(\overline{u},\overline{v})$ is a supersolution
for (\ref{System-log-perturbate}) for all $\epsilon \in (0,\epsilon _{0})$.

Then, owing to Theorem \ref{abs-theo} we conclude that the perturbed problem %
\eqref{System-log-perturbate} has a solution $(u_{\varepsilon
},v_{\varepsilon })\in W_{0}^{1,p(x)}(\Omega )\times W_{0}^{1,q(x)}(\Omega )$
within $[\underline{u},\overline{u}]\times \lbrack \underline{v},\overline{v}%
]$, for all $\epsilon \in (0,\varepsilon _{0}).$ Moreover, according to
Lemma \ref{alpha_beta_ineq} combined with (\ref{7.2}) and (\ref%
{super-sub-ineq}), we have that for $\sigma _{1},\sigma _{2}\in (0,1/N)$,
there are constants $C_{\sigma _{1}},C_{\sigma _{2}}>0$ such that 
\begin{equation}
\begin{array}{l}
-\gamma \log (v_{\varepsilon }+\varepsilon )+\theta v^{\alpha
(x)}_{\varepsilon}\leq \gamma (v_{\varepsilon }^{-\sigma _{1}}+C_{\sigma
_{1}}v_{\varepsilon }^{\sigma _{1}})+\theta v_{\varepsilon }^{\alpha (x)} \\ 
=v_{\varepsilon }^{-\sigma _{1}}(\gamma +\gamma C_{\sigma
_{1}}v_{\varepsilon }^{2\sigma _{1}}+\theta v^{\alpha (x)+\sigma
_{1}}_{\varepsilon}) \\ 
\leq \underline{v}^{-\sigma _{1}}(\gamma +\gamma C_{\sigma _{1}}\overline{v}%
^{2\sigma _{1}}+\theta \overline{v}^{\alpha (x)+\sigma _{1}}) \\ 
\leq (k_{0}d(x))^{-\sigma _{1}}(\gamma +\gamma C_{\sigma _{1}}\overline{v}%
^{2\sigma _{1}}+\theta \overline{v}^{\alpha (x)+\sigma _{1}})\leq
A_{1}d(x)^{-\sigma _{1}}\text{ in }\Omega%
\end{array}%
\end{equation}%
and%
\begin{equation}
\begin{array}{l}
-\gamma \log u_{\varepsilon }+u_{\varepsilon }^{\beta (x)}\leq \gamma
(u_{\varepsilon }^{-\sigma _{2}}+C_{\sigma _{2}}u_{\varepsilon }^{\sigma
_{2}})+\theta u_{\varepsilon }^{\beta (x)} \\ 
=\underline{u}^{-\sigma _{2}}(\gamma +\gamma C_{\sigma _{2}}\overline{u}%
^{2\sigma _{2}}+\theta \overline{u}^{\beta (x)+\sigma _{2}}) \\ 
\leq (k_{0}d(x))^{-\sigma _{2}}(\gamma +\gamma C_{\sigma _{2}}\overline{u}%
^{2\sigma _{2}}+\theta \overline{u}^{\beta (x)+\sigma _{2}})\leq
A_{2}d(x)^{-\sigma _{2}}\text{ in }\Omega ,%
\end{array}%
\end{equation}%
for some positive constants $A_{1}$ and $A_{2}$. Then, thanks to \cite[Lemma
2]{Alves-Moussaoui}, we deduce that $(u_{\varepsilon },v_{\varepsilon })\in
C^{1,\nu }(\overline{\Omega })\times C^{1,\nu }(\overline{\Omega }),$ for
certain $\nu \in (0,1)$.
\end{proof}

\begin{proof}[Proof of Theorem \protect\ref{T1}]
Set $\epsilon :=\frac{1}{n}$ for $n\geq 1/\varepsilon _{0}$. By Theorem \ref%
{perturbed-solutions}, we know that there exists a positive solution $%
(u_{n},v_{n}):=(u_{\frac{1}{n}},v_{\frac{1}{n}})$ bounded in $C^{1,\nu }(%
\overline{\Omega })\times C^{1,\nu }(\overline{\Omega }),$ for certain $\nu
\in (0,1)$, for problem%
\begin{equation}
\left\{ 
\begin{array}{ll}
-\Delta _{p(x)}u_{n}=-\gamma \log (|v_{n}|+\frac{1}{n})+
\theta|v_{n}|^{\alpha (x)} & \ \mbox{in}\ \Omega , \\ 
-\Delta _{q(x)}v_{n}=- \gamma\log (|u_{n}|+\frac{1}{n})+\theta|u_{n}|^{\beta
(x)} & \ \mbox{in}\ \Omega , \\ 
u_{n}=v_{n}=0 & \ \mbox{on}\ \partial \Omega .%
\end{array}%
\right.  \label{pn2}
\end{equation}%
Moreover, the property formulated in (\ref{super-sub-ineq}) holds true.
Employing Arzel\`{a}-Ascoli's theorem, we may pass to the limit in $C^{1}(%
\overline{\Omega })\times C^{1}(\overline{\Omega })$ and the limit functions 
$(u,v)\in C^{1}(\overline{\Omega })\times C^{1}(\overline{\Omega })$ satisfy
(\ref{p}) with $(u,v)\in \lbrack \underline{u},\overline{u}]\times \lbrack 
\underline{v},\overline{v}]$. This completes the proof.
\end{proof}

\section{Proof of Theorem \protect\ref{T2}}

\label{S5}

This section is devoted to the proof of Theorem \ref{T2}. For $\epsilon >0,$
let consider the regularized problem%
\begin{equation}
\left\{ 
\begin{array}{ll}
-\Delta _{p(x)}u=-\gamma \log (|v|+\epsilon )+\theta (|v|+\epsilon )^{\alpha
(x)} & in\text{ }\Omega , \\ 
-\Delta _{q(x)}v=-\gamma \log (|u|+\epsilon )+\theta (|u|+\epsilon )^{\beta
(x)} & in\text{ }\Omega , \\ 
u=v=0 & on\text{ }\Omega \text{.}%
\end{array}%
\right.  \label{System-perturbated}
\end{equation}%
Our demonstration strategy will be to show- by applying the well known
result due to Rabinowitz \cite{Rabinowitz}- that, for each $\lambda >0,$
system \eqref{System-perturbated} possesses a positive solution $%
(u_{\epsilon },v_{\epsilon })$ in $W_{0}^{1,p(x)}(\Omega )\times
W_{0}^{1,q(x)}(\Omega )$, and then derive a solution of (\ref{p}) by taking
the limit $\epsilon \rightarrow 0$.\newline

\subsection{\textbf{Existence result for the regularized system}}

Fix $\epsilon >0$ and for each pair $(f,g)\in L^{p^{\prime }(x)}(\Omega
)\times L^{q^{\prime }(x)}(\Omega ),$ let consider the auxiliary problem%
\begin{equation}
\left\{ 
\begin{array}{l}
-\Delta _{p(x)}u=\lambda (-\gamma \log (|g|+\epsilon )+\theta (|g|+\epsilon
)^{\alpha (x)})\ \mbox{in}\ \Omega , \\ 
-\Delta _{q(x)}v=\lambda (-\gamma \log (|f|+\epsilon )+\theta (|f|+\epsilon
)^{\beta (x)})\ \mbox{in}\ \Omega , \\ 
u=v=0\ \mbox{on}\ \partial \Omega .%
\end{array}%
\right.  \label{System-perturbated-fg}
\end{equation}%
Observe that:

\begin{itemize}
\item $-\log (|g|+\epsilon )\in L^{p^{\prime }(x)}(\Omega )$: Indeed,
consider $\theta >0$ such that $0<\theta (p^{\prime })^{+}\leq q^{\prime
}(x) $ for all $x\in \overline{\Omega }$. By Lemma \ref{alpha_beta_ineq},
one has%
\begin{equation*}
|\log (|g(x)|+\epsilon )|^{p^{\prime }(x)}\leq (|g(x)|^{\theta p^{\prime
}(x)}+C)\leq ((1+|g(x)|)^{\theta p^{\prime }(x)}+C).
\end{equation*}%
From $(\ref{H1})$ the claim follows.

\item $(|g(x)|+\epsilon )^{\alpha (x)}\in L^{p^{\prime }(x)}$: By $(\ref{H1}%
) $ notice that%
\begin{equation*}
(|g(x)|+\epsilon )^{\alpha (x)p^{\prime }(x)}\leq (|g(x)|+1)^{\alpha
(x)p^{\prime }(x)}\leq (|g(x)|+1)^{q^{\prime }(x)}.
\end{equation*}%
Since $W^{1,q(x)}(\Omega )\hookrightarrow L^{q^{\prime }(x)}(\Omega )$ the
claim is proved.
\end{itemize}

In the same manner we have $|\log (|f|+\epsilon )|\in L^{q^{\prime }}(\Omega
)$ and $(|f|+\epsilon )^{\beta (x)}\in L^{q^{\prime }(x)}$ for all $f\in
L^{p^{\prime }(x)}(\Omega )$. Then, on account of the above remarks, the
unique solvability of $(u,v)\in W_{0}^{1,p(x)}(\Omega )\times
W_{0}^{1,q(x)}(\Omega )$ in (\ref{System-perturbated-fg}) is readily derived
from Minty-Browder Theorem. Therefore, the solution operator 
\begin{equation*}
\mathcal{T}:\mathbb{R}^{+}\times L^{p^{\prime }(x)}(\Omega )\times
L^{q^{\prime }(x)}(\Omega )\rightarrow W_{0}^{1,p(x)}(\Omega )\times
W_{0}^{1,q(x)}(\Omega )
\end{equation*}%
is well defined.

\begin{lemma}
\label{L1}The operator $\mathcal{T}:\mathbb{R}^{+}\times L^{p^{\prime
}(x)}(\Omega )\times L^{q^{\prime }}(\Omega )\rightarrow L^{p^{\prime
}(x)}(\Omega )\times L^{q^{\prime }(x)}(\Omega )$ is continuous and compact.
\end{lemma}

\begin{proof}
Consider a sequence $(\lambda _{n},f_{n},g_{n})\rightarrow (\lambda ,f,g)$
in $\mathbb{R}^{+}\times L^{p^{\prime }(x)}(\Omega )\times L^{q^{\prime
}(x)}(\Omega )$ and $(u_{n},v_{n}):=\mathcal{T}(\lambda _{n},f_{n},g_{n}).$
Using $u_{n}$ as a test function, one gets%
\begin{equation}
\begin{array}{l}
\int_{\Omega }|\nabla u_{n}|\ dx=\lambda _{n}\left( \int_{\Omega }-\gamma
u_{n}\log (|g_{n}|+\epsilon )+\theta u_{n}(|g_{n}|+\epsilon )^{\alpha
(x)}\right) \ dx \\ 
\leq C\Vert u_{n}\Vert _{L^{p(x)}(\Omega )}\Vert \log (|g_{n}|+\epsilon
)\Vert _{L^{p^{\prime }(x)}(\Omega )}+C\Vert u_{n}\Vert _{L^{p(x)}(\Omega
)}\Vert (|g_{n}|+\epsilon )^{\alpha (x)}\Vert _{L^{p^{\prime }(x)}(\Omega )}.%
\end{array}
\label{f-conv}
\end{equation}%
Since $\{g_{n}\}$ is bounded in $L^{q\prime (x)}(\Omega )$ by Lemma \ref%
{alpha_beta_ineq}, $\{u_{n}\}$ is bounded in $W^{1,p(x)}(\Omega ).$ Let $%
(u,v):=\mathcal{T}(\lambda ,f,g).$ Using $u_{n}-u$ as a test function we
have 
\begin{equation}
\begin{array}{l}
\int_{\Omega }\langle |\nabla u_{n}|^{p(x)-2}\nabla u_{n}-|\nabla
u|^{p(x)-2}\nabla u,\nabla (u_{n}-u)\rangle \ dx \\ 
=\gamma \int_{\Omega }(\lambda \log (|g|+\epsilon )-\lambda \log
(|g_{n}|+\epsilon ))(u_{n}-u)\ dx \\ 
\text{ \ \ }+\theta \int_{\Omega }(\lambda (|g|+\epsilon )^{\alpha
(x)}-\lambda _{n}(|g_{n}|+\epsilon )^{\alpha (x)})(u_{n}-u)\ dx.%
\end{array}
\label{conv2-u_n}
\end{equation}%
Note that 
\begin{equation}
\begin{aligned} & \left|\gamma \int_{\Omega} (\lambda \log(|g|+\epsilon) -
\lambda_n \log(|g_n|+\epsilon))(u_n-u) \ dx \right|  \leq C \| u_n - u\|_{L^{p(x)}(\Omega)} \\ &\times  \| (\lambda
\log(|g|+\epsilon) - \lambda_n \log(|g_n|+\epsilon))
\|_{L^{p^{'}(x)}(\Omega)} ,
\end{aligned}  \label{f-conv-2}
\end{equation}%
where the constant $C$ does not depend on $n\in \mathbb{N}.$

In the sequel, up to a subsequence, we can assume that $g_{n}(x)\rightarrow
g(x)$ a.e in $\Omega $ and $|g_{n}(x)|\leq h$ a.e in $\Omega $ for some $%
h\in L^{q^{^{\prime }}(x)}(\Omega ).$ Then, by Lemma \ref{alpha_beta_ineq}
and the Lebesgue Theorem, we have 
\begin{equation}
\gamma \int_{\Omega }(\lambda \log (|g|+\epsilon )-\lambda _{n}\log
(|g_{n}|+\epsilon ))(u_{n}-u)\ dx\rightarrow 0.  \label{f-conv-3}
\end{equation}%
A similar reasoning leads to 
\begin{equation}
\theta \int_{\Omega }(\lambda (|g|+\epsilon )^{\alpha (x)}-\lambda
_{n}(|g_{n}|+\epsilon )^{\alpha (x)})(u_{n}-u)\ dx\rightarrow 0.
\label{f-conv-4}
\end{equation}%
Since $\{u_{n}\}$ is bounded in $W_{0}^{1,p(x)}(\Omega )$, from %
\eqref{conv2-u_n} we deduce that $u_{n}\rightarrow u$ in $%
W_{0}^{1,p(x)}(\Omega ).$ This proves that $\mathcal{T}$ is continuous.

In order to show that $\mathcal{T}$ is compact, it sufficies to prove that $%
\overline{\mathcal{T}(U)}$ is compact for all $U\subset E$ bounded. At this
point, a quite similar argument as above produces the desired conclusion.
This completes the proof.
\end{proof}

\begin{theorem}
\label{T5}Under assumptions (\ref{H1}) and (\ref{H2'}), problem (\ref%
{System-perturbated}) admits a solution $(u_{\epsilon },v_{\epsilon })$ for
all $\epsilon >0$.
\end{theorem}

\begin{proof}
From Lemma \ref{L1} and invoking \cite{Rabinowitz}, there is an unbounded
continuum $\mathcal{C}$ of solutions of the equation $(u,v)=\mathcal{T}%
(\lambda ,u,v)$, that is, $(\lambda ,u,v)\in \mathcal{C}$ is a solution of (%
\ref{System-perturbated}).

On the other hand, by Lemma \ref{alpha_beta_ineq} the function $f(x)=\theta
x^{\delta }-\gamma \log x,$ for $x>0$, attains a strictly positive minimum
if $\delta >\frac{\gamma }{\theta e}$. Since $\alpha ^{-},\alpha ^{+}>\frac{%
\gamma }{\theta e},$ it follows that

\begin{itemize}
\item if $|u|+\epsilon \geq 1,$ then 
\begin{equation*}
-\gamma \log (|u|+\epsilon )+\theta (|u|+\epsilon )^{\alpha (x)}\geq -\gamma
\log (|u|+\epsilon )+\theta (|u|+\epsilon )^{\alpha _{-}}\geq m_{-}>0
\end{equation*}%
where $m_{-}=\min \{-\gamma \log x+\theta x^{\alpha _{-}},x>0\}.$

\item if $|u|+\epsilon <1,$ then 
\begin{equation*}
-\log (|u|+\epsilon )+(|u|+\epsilon )^{\alpha (x)}\geq -\log (|u|+\epsilon
)+(|u|+\epsilon )^{\alpha _{+}}\geq m_{+}>0
\end{equation*}%
where $m_{+}=\min \{-\gamma\log x+\theta x^{\alpha _{+}},x>0\}.$
\end{itemize}

Therefore, $-\Delta _{p(x)}u\geq m_{1}>0$ where $m_{1}=\min \{m_{-},m_{+}\}$
and, with a quite similar reasoning, we get $-\Delta _{q(x)}v\geq m_{2}>0$
for some $m_{2}>0.$ Thus, by maximum principle, $\mathcal{C}\setminus
\{(0,0,0)\}$ must be constituted by strictly positive functions.

Next, we show that the component $\mathcal{C}$ is unbounded with respect to $%
\lambda \geq 0$. By contradiction, suppose that there is $\lambda ^{\star
}>0 $ such that $(\lambda ,u,v)\in \mathcal{C}$ implies that $\lambda \leq
\lambda ^{\star }.$ Fix $0<\overline{\gamma }\leq \left( \frac{q^{\prime }}{%
p^{\prime }}\right) ^{-}.$ Using $u$ as a test function we get 
\begin{equation}
\int_{\Omega }|\nabla u|^{p(x)}dx=\lambda \left( \int_{\Omega }-\gamma u\log
(|v|+\varepsilon )dx+\int_{\Omega }\theta (|v|+\varepsilon )^{\alpha
(x)}udx\right)  \label{comp-1}
\end{equation}%
and%
\begin{equation}
\begin{array}{l}
\int_{\Omega }|-\gamma u\log (|v|+\varepsilon )|dx\leq \int_{\Omega }\gamma
|u|((|v|+\varepsilon )^{-\overline{\gamma }}+C_{\gamma ,\overline{\gamma }%
}(|v|+\varepsilon )^{\overline{\gamma }})dx \\ 
\leq C\Vert u\Vert _{L^{p(x)}(\Omega )}+C\int_{\Omega }|u||v|^{\overline{%
\gamma }}dx,%
\end{array}
\label{comp-2}
\end{equation}%
where $C$ depends on $\lambda ^{\ast },\varepsilon ,\gamma $ and $\overline{%
\gamma }$. Note that 
\begin{equation}
\begin{array}{l}
\int_{\Omega }\theta |u|(|v|+\varepsilon )^{\alpha (x)}dx\leq \int_{\Omega
}\theta |u|(|v|+1)^{\alpha (x)}dx \\ 
\leq C\Vert u\Vert _{L^{p(x)}(\Omega )}+C\int_{\Omega }|u||v|^{\alpha (x)}dx,%
\end{array}
\label{comp-3}
\end{equation}%
where $C$ depends on $\theta $ and $\varepsilon .$ Now we will estimate the
integral $\int_{\Omega }|u||v|^{\alpha (x)}\ dx$. We have $|v(x)|^{\alpha
(x)}\in L^{{p^{^{\prime }}(x)}}(\Omega )$. In order to prove this, note that%
\begin{equation*}
|v(x)|^{\alpha (x)p^{^{\prime }}(x)}\leq (1+|v(x)|)^{\alpha (x)p^{^{\prime
}}(x)}\leq (1+|v(x)|)^{q^{^{\prime }}(x)}.
\end{equation*}%
The last function belongs to $L^{1}(\Omega )$ because $W^{1,q(x)}(\Omega
)\hookrightarrow L^{q^{\star }(x)}(\Omega )\hookrightarrow L^{q^{^{\prime
}}(x)}(\Omega ).$ Thus by H\"{o}lder inequality we obtain 
\begin{equation*}
\int_{\Omega }|v|^{\alpha (x)}|u|\ dx\leq C\Vert |v|^{\alpha (x)}\Vert
_{L^{p^{^{\prime }}(x)}(\Omega )}\Vert u\Vert _{W^{1,p(x)}(\Omega )}.
\end{equation*}%
By H\"{o}lder inequality and considering all the possibilities for the norms \newline
$\Vert |v|^{\alpha (x)}\Vert _{L^{p^{^{\prime }}}(\Omega )}$ and $\Vert
|v|^{\alpha (x)}\Vert _{L^{\frac{q^{^{\prime }}(x)}{\alpha (x)p^{^{\prime
}}(x)}}(\Omega )}$, we get 
\begin{small}
\begin{equation}
\begin{array}{l}
\Vert |v|^{\alpha (x)}\Vert _{L^{p^{^{\prime }}(x)}(\Omega )}\leq C\Vert
|v|^{\alpha (x)p^{^{\prime }}(x)}\Vert _{L^{\frac{q^{^{\prime }}(x)}{\alpha
(x)p^{^{\prime }}(x)}}(\Omega )}^{\frac{1}{(p^{^{\prime }})^{+}}}+1+C\Vert
|v|^{\alpha (x)p^{^{\prime }}(x)}\Vert _{L^{\frac{q^{^{\prime }}(x)}{\alpha
(x)p^{^{\prime }}(x)}}(\Omega )}^{\frac{1}{(p^{^{\prime }})^{-}}} \\ 
\leq C\left( \left( \int_{\Omega }|v|^{q^{^{\prime }}(x)}\ dx\right) ^{\frac{%
1}{(p^{^{\prime }})^{+}}\left( \frac{\alpha (x)p^{^{\prime }}(x)}{%
q^{^{\prime }}(x)}\right) ^{+}}+1+\left( \int_{\Omega }|v|^{q^{^{\prime
}}(x)}\ dx\right) ^{\frac{1}{(p^{^{\prime }})^{+}}\left( \frac{\alpha
(x)p^{^{\prime }}(x)}{q^{^{\prime }}(x)}\right) ^{-}}\right) \\ 
+C\left( \left( \int_{\Omega }|v|^{q^{^{\prime }}(x)}\ dx\right) ^{\frac{1}{%
(p^{^{\prime }})^{-}}\left( \frac{\alpha (x)p^{^{\prime }}(x)}{q^{^{\prime
}}(x)}\right) ^{+}}+1+\left( \int_{\Omega }|v|^{q^{^{\prime }}(x)}\
dx\right) ^{\frac{1}{(p^{^{\prime }})^{-}}\left( \frac{\alpha (x)p^{^{\prime
}}(x)}{q^{^{\prime }}(x)}\right) ^{-}}\right) .%
\end{array}
\label{comp-4}
\end{equation}%
\end{small}
Using the embedding $W^{1,q(x)}(\Omega )\hookrightarrow L^{q^{^{\prime
}}(x)}(\Omega )$, considering all the possibilities for the norms $\Vert
v\Vert _{L^{q^{^{\prime }}(x)}(\Omega )}$, $\Vert \nabla u\Vert
_{L^{p(x)}(\Omega )}$, the estimates (\ref{comp-1}), (\ref{comp-2}), (\ref%
{comp-3}), (\ref{comp-4}) and repeating the arguments for the integral $%
\int_{\Omega }|u||v|^{\gamma }\ dx$ we obtain 
\begin{align*}
\Vert u& \Vert _{W^{1,p(x)}(\Omega )}\leq C{\Vert u\Vert _{W^{1,p(x)}(\Omega
)}^{\frac{1}{p^{-}}}}(1+\Vert v\Vert _{W^{1,q(x)}(\Omega )})^{\frac{\alpha
^{+}}{p^{-}}} \\
& +C{\Vert u\Vert _{W^{1,p(x)}(\Omega )}}^{\frac{1}{p^{+}}}(1+\Vert v\Vert
_{W^{1,q(x)}(\Omega )})^{\frac{\alpha ^{+}}{p^{-}}}+C\Vert u\Vert
_{W^{1,p(x)}(\Omega )}^{\frac{1}{p^{-}}} \\
& C\Vert u\Vert _{W^{1,p(x)}(\Omega
)}^{\frac{1}{p^{+}}}.
\end{align*}%
Thus, 
\begin{equation}
\begin{array}{l}
\Vert u\Vert _{W^{1,p(x)}(\Omega )}^{p^{-}}\leq C\Vert u\Vert
_{W^{1,p(x)}(\Omega )}(1+\Vert v_{n}\Vert _{W^{1,p(x)}(\Omega )})^{\alpha
^{+}}+C(1+\Vert u\Vert _{W^{1,p(x)}(\Omega )}).%
\end{array}
\label{lead-1}
\end{equation}%
A similar reasoning leads to 
\begin{equation}
\begin{array}{l}
\Vert v\Vert _{W^{1,p(x)}(\Omega )}^{q^{-}}\leq C\Vert v\Vert
_{W^{1,q(x)}(\Omega )}(1+\Vert u\Vert _{W^{1,p(x)}(\Omega )})^{\beta
^{+}}+C(1+\Vert v\Vert _{W^{1,q(x)}(\Omega )}).%
\end{array}
\label{lead-2}
\end{equation}%
Since $\alpha ^{+}+1<p^{-}$ and $\beta ^{+}+1<q^{-}$ it follows that the
component $\mathcal{C}$ is bounded, which is absurd. Consequently, $\mathcal{%
C}$ crosses the set $\{1\}\times L^{p^{^{\prime }}(x)}(\Omega )\times
L^{q^{^{\prime }}(x)}(\Omega )$ and this implies that there is a solution $%
(u_{\epsilon },v_{\epsilon })$ of (\ref{System-perturbated}). The proof is
completed.
\end{proof}

\subsection{\textbf{Passage to the limit}}

Set $\epsilon =\frac{1}{n}$ in (\ref{System-perturbated}) with any integer $%
n\geq 1$. By applying Theorem \ref{T5}, we know that there exist $u_{\frac{1%
}{n}}:=u_{n}$ and $v_{\frac{1}{n}}:=v_{n}$ that solve the problem (\ref%
{System-perturbated}) with $\epsilon =\frac{1}{n}$. \newline

\noindent \textbf{Claim:} The sequences $\{u_{n}\}$ and $\{v_{n}\}$ are
bounded in $W_{0}^{1,p(x)}(\Omega )$ and $W_{0}^{1,q(x)}(\Omega ),$
respectively and the weak limits (that exist up to a subsequence) are
strictly positive in $\Omega .$

First of all, we know that $-\Delta _{p(x)}u_{n}\geq m_{1}>0$ where $%
m_{1}=\min \{m_{-},m_{+}\}$. If $w_{1}$ denotes the unique positive solution
of 
\begin{equation*}
\left\{ 
\begin{array}{ll}
-\Delta _{p(x)}w_{1}=m_{1} & \ \mbox{in}\ \Omega , \\ 
w_{1}=0 & \ \mbox{on}\ \partial \Omega ,%
\end{array}%
\right.
\end{equation*}%
the maximum principle gives 
\begin{equation*}
u_{n}\geq w_{1}>0\quad \mbox{in}\quad \Omega .
\end{equation*}%
By the strong maximum principle (see \cite[Theorem 1.2]{FZZ}) we have $\frac{%
w_{1}}{\partial \eta }>0$ where $\eta $ is the inward normal vector of $%
\partial \Omega .$ Let $\phi _{q^{-}}$ an eigenfunction associated to the
first eigenvalue of the operator $(-\Delta _{q^{-}},W_{0}^{1,q^{-}}(\Omega
)) $. Note that 
\begin{equation}
P_{1}\phi _{q^{-}}(x)\leq w_{1}(x)
\end{equation}%
where $P_{1}$ is a positive constant that does not depend on $x\in \Omega .$

Denote $\phi _{p^{-}}$ an eigenfunction associated to the first eigenvalue
of the operator $(-\Delta _{p^{-}},W_{0}^{1,p^{-}}(\Omega ))$. Reasoning as
above, we also have $v_{n}\geq w_{2}>0$ and 
\begin{equation}
L_{1}\phi _{p^{-}}(x)\leq w_{2}(x)  \label{comparable-2}
\end{equation}%
where $L_{1}$ is a positive constant that does not depend on $x\in \Omega .$
with $w_{2}$ the unique positive solution of 
\begin{equation*}
\left\{ 
\begin{array}{ll}
-\Delta _{q(x)}w_{2}=m_{2} & \ \mbox{in}\ \Omega , \\ 
w_{2}=0 & \ \mbox{on}\ \partial \Omega 
\end{array}%
\right. 
\end{equation*}%
Let $\delta \in (0,1)$. By using $u_{n}$ as a test function in its
corresponding system of equations and arguing as in the set of inequalities (%
\ref{comp-2}) and (\ref{comp-3}) we get{\small 
\begin{align*}
\int_{\Omega }|\nabla u_{n}|^{p(x)}\ dx & \leq C\left( \int_{\Omega }\frac{%
|u_{n}|}{|v_{n}|^{\delta }}\ dx+\int_{\Omega }|u_{n}||v_{n}|^{\delta }\
dx+\Vert u_{n}\Vert _{L^{p(x)}(\Omega )} \right)\\
 &+ C\int_{\Omega}|u_{n}||v_{n}|^{\alpha (x)}\ dx ,
\end{align*}%
} where $C$ is a constant that depends on $\gamma $ and $\theta .$
Hardy-Sobolev inequality (see \cite{Kavian-thesis}), together with the
embedding $W_{0}^{1,p(x)}(\Omega )\hookrightarrow W_{0}^{1,p^{-}}(\Omega )$
and the relation (\ref{comparable-2}), it follows that 
\begin{equation}
\begin{aligned} \int_{\Omega} \frac{|u_n|}{|v_n|^{\delta}} \ dx &\leq
\int_{\Omega} \frac{|u_n|}{w_{2}^{\delta}} \ dx \leq \int_{\Omega}
\frac{|u_n|}{C \phi_{p^{-}}^{\delta}} \ dx \\ & \leq C \| \nabla
u_n\|_{L^{p(x)}(\Omega)}. \end{aligned}  \label{h-s-app}
\end{equation}%
where the constant $C$ does not depend on $n\in \mathbb{N}.$

By (\ref{h-s-app}) and using the reasoning that leads to (\ref{lead-1}) and (%
\ref{lead-2}), we obtain that $(u_{n},v_{n})$ is bounded in $E.$ Passing to
a subsequence we have

\begin{itemize}
\item $u_n \rightharpoonup u$ in $W^{1,p(x)}(\Omega),$

\item $u_n \rightarrow u$ in $L^{p(x)}(\Omega),$

\item $u_n \rightarrow u $ a.e in $\Omega,$

\item $v_n \rightharpoonup v$ in $W^{1,p(x)}(\Omega),$

\item $v_n \rightarrow v$ in $L^{p(x)}(\Omega),$

\item $v_{n}\rightarrow v$ a.e in $\Omega ,$
\end{itemize}

for some pair $(u,v)\in E.$ From the previous pointwise convergence and the
relations between $w_{1},$ $u_{n}$ and $w_{2},$ $v_{n},$ we conclude that $%
u>0$ and $v>0,$ which proves the claim.

\bigskip

Taking $u_{n}$ as a test function and repeating the arguments of the
relations \eqref{f-conv}-\eqref{f-conv-4}, we get that $u_{n}\rightarrow u$
in $W_{0}^{1,p(x)}(\Omega )$. Notice that the same argument provides that $%
v_{n}\rightarrow v$ in $W_{0}^{1,q(x)}(\Omega ).$

From the previous strong convergence of $u_{n}$ and $v_{n}$, combined with
the Lebesgue's Dominated Convergence Theorem, we obtain 
\begin{equation*}
\int_{\Omega }|\nabla u|^{p(x)-2}\nabla u\nabla \phi \ dx=\int_{\Omega
}-\gamma \phi \log v\ dx+\int_{\Omega }\theta \phi v\ dx
\end{equation*}%
and 
\begin{equation*}
\int_{\Omega }|\nabla v|^{p(x)-2}\nabla u\nabla \psi \ dx=\int_{\Omega
}-\gamma \psi \log u\ dx+\int_{\Omega }\theta \psi u\ dx
\end{equation*}%
for all $(\phi ,\psi )\in W_{0}^{1,p(x)}(\Omega )\times
W_{0}^{1,q(x)}(\Omega )$ and the existence of solution is proved. 
\section*{Acknowledgements} The work was started while the second and the
third author were visiting the Federal University of Campina Grande. They
thank professor Claudianor Alves and the other members of the department for
hospitality. 


\begin{thebibliography}{99}
\bibitem{Acerbi1} E. Acerbi \& G. Mingione, \emph{Regularity results for
stationary electrorheological fluids}, \textit{Arch. Rational Mech. Anal.}
164\textbf{\ }(2002), 213-259.

\bibitem{Acerbi2} E. Acerbi \& G. Mingione, \emph{Regularity results for
electrorheological fluids: stationary case}, \textit{C.R. Math. Acad. Sci.
Paris} 334\textbf{\ }(2002), 817-822.


\bibitem{Alves2} C.O. Alves, \emph{Existence of solutions for a degenerate }$%
p(x)$\emph{-Laplacian equation in }$\mathbb{R}^{N}$, \textit{J. Math. Anal.
Appl.} 345\textbf{\ }(2008), 731-742.


\bibitem{AlvesBarreiro} C.O. Alves \& J.L.P. Barreiro, \emph{Existence and
multiplicity of solutions for a }$p(x)$\emph{-Laplacian equation with
critical growth}, \textit{J. Math. Anal. Appl.} 403 (2013), 143-154.


\bibitem{Alves-Correa} C. O. Alves \& J. S. A. Corr\^{e}a, \emph{On the
existence of positive solution for a class of singular systems involving
quasilinear operators}\textit{,} \textit{Appl. Math. Comput.} 185 (2007),
no. 1, 727-736.

\bibitem{AlvesFerreira1} C.O. Alves \& M.C. Ferreira, \emph{Existence of
solutions for a class of }$p(x)$\emph{-Laplacian equations involving a
concave-convex nonlinearity with critical growth in }$\mathbb{R}^{N}$, 
\textit{Topol. Methods Nonl. Anal.} 45 (2015), no. 2, 399-422.

\bibitem{alves-moussaoui} C. O. Alves \& A. Moussaoui, \emph{Existence of
solutions for a class of singular elliptic systems with convection term}%
\textit{,} \textit{Asymptot. Anal.} 90 (2014), no. 3-4, 237-248.


\bibitem{AlvesSouto} C.O. Alves \& M.A.S. Souto, \emph{Existence of
solutions for a class of problems in }$\mathbb{R}^{N}$\emph{\ involving }$%
p(x)$\emph{-Laplacian}, \textit{Prog. Nonl. Diff. Eqts. and their Appl.} 66%
\textbf{\ }(2005), 17-32.

\bibitem{Alves-Moussaoui} C. O. Alves \& A. Moussaoui, \emph{Existence and
regularity of solutions for a class of singular }$(p(x),q(x))$\emph{-
Laplacian systems,} To appear in Complex Var. Elliptic Equ.

\bibitem{Antontsev} S.N. Antontsev \& J.F. Rodrigues, \emph{On stationary
thermo-rheological viscous flows}, \textit{Ann. Univ. Ferrara Sez. VII Sci.
Mat.} 52 (2006), 19-36.

\bibitem{carl} S. Carl, V. K. Le and D. Motreanu, \emph{Nonsmooth
variational problems and their inequalities}. \emph{Comparaison principles
and applications}, \textit{Springer, New York,} 2007.

\bibitem{CLions} A. Chambolle \& P.L. Lions, \emph{Image recovery via total
variation minimization and related problems}, \textit{Numer. Math.} 76
(1997), 167-188.

\bibitem{Chen} Y. Chen, S. Levine \& M. Rao, \emph{Variable exponent, linear
growth functionals in image restoration}, \textit{SIAM J. Appl. Math.} 66%
\textbf{\ }(2006), 1383-1406.

\bibitem{Queiroz} O.S. de Queiroz, \emph{A Neumann problem with logarithmic
nonlinearity in a ball}\textit{,} \textit{Nonl. Anal}. 70 (2009), no. 4,
1656-1662.





\bibitem{FZZ2} {X. Fan, Q. Zhang \& D. Zhao,} \emph{Eigenvalues of }$p(x)-$%
\emph{Laplacian Dirichlet problem}\textit{,} J. Math. Anal. Appl. 302
(2005), 306-317.

\bibitem{Fan1} X.L. Fan, \emph{On the sub-supersolution method for }$p(x)$%
\emph{-Laplacian equations}, \textit{J. Math. Anal. Appl.} 330\textbf{\ }%
(2007), 665-682.


\bibitem{FanZhao0} X. Fan \& D. Zhao, \emph{A class of De Giorgi type and H%
\"{o}lder continuity}, \textit{Nonl. Anal.} 36\textbf{\ }(1999), 295-318.


\bibitem{FZZ} X.L. Fan, Y.Z. Zhao \& Q.H. Zhang, \emph{A strong maximum
principle for p(x)-Laplace equations,} \textit{Chinese J. Contemp. Math.} 
\textbf{(}24) (2003), no. 3, 1-7.

\bibitem{FanShenZhao} X.L. Fan, J.S. Shen \& D. Zhao, \emph{Sobolev
embedding theorems for spaces }$W^{k,p(x)}\big(\Omega \big)$, \textit{J.
Math. Anal. Appl.} 262\textbf{\ }(2001), 749-760.

\bibitem{BSS} J. Fern\'{a}ndez Bonder, N. Saintier \& A. Silva. \emph{On the
Sobolev embedding theorem for variable exponent spaces in the critical range}%
, \textit{J. Diff. Eqts.} 253 (2012), 1604-1620.




\bibitem{Ferreira-Queiroz} L. C. Ferreira \& O. S. de Queiroz, \emph{A
singular parabolic equation with logarithmic nonlinearity and L}$^{p}$\emph{%
-initial data}\textit{,} \textit{J. Diff. Eqts} 249 (2010), no. 2, 349-365.


\bibitem{Kavian-thesis} O. Kavian, \emph{Inegalit\'{e} de HardySobolev et
applications,} \textit{Th\'{e}se de Doctorate de 3eme cycle, Universit\'{e}
de Paris }VI (1978).

\bibitem{Lazer-Mckenna} A. C. Lazer \& P. J. Mckenna, \emph{On a singular
nonlinear elliptic boundary-value problem}\textit{,} \textit{Proc. American
Math. Soc.} 3 (111), 1991.

\bibitem{MR} M. Mih$\breve{a}$ilescu \& V. R$\breve{a}$dulescu,\emph{\ On a
nonhomogeneous quasilinear eigenvalue problem in Sobolev spaces with
variable exponent}, \textit{Proc. Amer. Math. Soc.} 135(9) (2007), 2929-2937.

\bibitem{MY} J. Mo \& Z. Yang, \emph{Existence of solutions to }$p$\emph{%
-laplace equations with logarithmic nonlinearity}, \textit{Electronic J.
Diff. Eqts.} 87 (2009), 1-10.

\bibitem{MQ} M. Montenegro \& O. S. de Queiroz, \emph{Existence and
regularity to an elliptic equation with logarithmic nonlinearity}, \textit{%
J. Diff. Eqts} 246\textbf{\ (2)} (2009), 482-511.

\bibitem{Necas-book} J. Necas,\emph{\ Introduction to the Theory of
Nonlinear Elliptic Equations}, John Willey \& Sons (1983)

\bibitem{Rabinowitz} P.H. Rabinowitz, \emph{Some global results for
nonlinear eigenvalue problems}\textit{,} \textit{J. Funct. Anal. }\textbf{7}
(1971) 487-513.

\bibitem{R*} V. R$\breve{a}$dulescu, \emph{Nonlinear elliptic equations with
variable exponent: old and new}, \textit{Nonl. Anal. }121 (2015), 336-369.

\bibitem{RR*} V. R$\breve{a}$dulescu \& D. Repov$\breve{s}$, \emph{Partial
differential equations with variable exponents. Variational methods and
qualitative analysis.} \textit{Monographs and Research Notes in Mathematics.
CRC Press, Boca Raton}, FL, 2015.

\bibitem{Ru} M. Ruzicka, \emph{Electrorheological fluids: Modeling and
mathematical theory}. Lecture Notes in Math., vol. 1748, Springer-Verlag,
Berlin (2000).

\bibitem{Re*} D. Repov$\breve{s}$, \emph{Stationary waves of Schr\"{o}%
dinger-type equations with variable exponent}, \textit{Anal. Appl. (Singap.}%
) 13 (2015), no. 6, 645-661.

\bibitem{Salin} T. Salin, \emph{On quenching with logarithmic singularity}, 
\textit{Nonl. Anal.} 52 (2003) 261-289.

\bibitem{YY} H. Yin \& Z. Yang, \emph{Existence and asymptotic behavior of
positive solutions for a class of (p(x),q(x))-Laplacian systems}, Diff.
Eqts. App. 6 (3) (2014), 403-415.

\bibitem{QZ} Q. Zhang, \emph{Existence and asymptotic behavior of positive
solutions to }$p(x)$\emph{-Laplacian equations with singular nonlinearities}
, J. Inequalities and App. (2007), DOI: 10.1155/2007/19349.
\end{thebibliography}
\end{document}